\documentclass[12pt]{amsart}

\usepackage{amssymb}
\usepackage{amsmath}

\usepackage[pdftex]{graphicx}
\newtheorem{theorem}{Theorem}
\newtheorem{lemma}{Lemma}
\begin{document}


\centerline{\bf\large On the generalized linear and non-linear DFC}
\centerline{\bf\large in non-linear dynamics}

\bigskip

\centerline{\bf D.Dmitrishin, A.Khamitova and A.Stokolos}

\bigskip

\textbf{Abstract} The article is devoted to investigation of robust stability of the generalized linear control of the  discrete autonomous dynamical systems. Sharp necessary conditions on the size of the set of multipliers that guaranty robust stabilization of the equilibrium of the system are provided. Surprisingly enough it turns out that the generalized linear delayed feedback control has same limitation as the classical Pyragas DFC. This generalized Ushio 1996 DFC limitation statement. Note that in scalar case a generalized non-linear control can robustly stabilize an equilibrium for any admissible range of multipliers \cite{[12]}. In the current article similar result is obtained in the vector-valued setting. The article is an updated version of \cite{DKSS}.  

\section{Introduction.}
The problem of optimal impact on a chaotic mode is one of the most fundamental in nonlinear dynamics \cite{[1],[2]}. To solve it
various schemes where proposed. Some of them are based on a special representation of the delayed feedback (DFC) \cite{[3]} that allows to stabilize a-priori not known equilibriums or cycles. Despite simplicity of implementation of DFC scheme, it does have  restrictions in application, that are connected with its linearity and the use only one previous state \cite{[4]}. The classical scheme of DFC works only for a limited region in the spaces of the parameters of the initial nonlinear system. To increase the area of applicability of the DFC  various generalizations of the classical scheme were suggested. E.g. in \cite{[5]} the control involves the information about previous states; in \cite{[6]} the nonlinear scheme of DFC with one delay  was considered and advantages at such modification were discussed. In particular, the control became robust. In \cite{[7]} a mixed linear-nonlinear DFC scheme was investigated; in \cite{[8]} the so-called, predictive DFC scheme was studied. In \cite{[9]} the ideas of \cite{[4],[6]} were synthesized and the most intrinsic control - the non-linear scheme with several delays - was proposed
\begin{equation}\label{nlc}
u=-\sum _{j=1}^{N-1}\varepsilon _{j} \left(F_h\left(x_{n-j+1} \right)-F_h\left(x_{n-j} \right)\right),
\end{equation}
where strengthening coefficients are small, e.g. $\left|\varepsilon _j \right|<1, j=1,...,N-1.$

Despite of simplicity of the classical and generalized DFC schemes, analytical
investigation of the closed system is a difficult task. The complexity is caused
by a geometry of the canonical region of Schur stability for polynomials in the space coefficients \cite{[10]}. Thus only numerical and experimental results relating to properties and to applicability of DFC schemes are known. In general, the problem obtaining sufficient conditions that guarantee applicability of various algorithms still open.

For a scalar case of the non-linear DFC with several delays the problem is completely solved in \cite{[9]}, where it was reduced by means of harmonic analysis to the problem if linear optimization.

For the classical linear  DFC with several
delays it is impossible to reduce the same problem to linear optimization even in a scalar case. In
this case the methods of complex analysis connected with
properties of mappings of the unit disc of the complex plane turns out to be effective.
These methods were extended from a class of univalent functions to any analytical functions in the disc.
It was discovered that the use of several delays does not give any
advantages in comparison with one delay. Notice that
for the non-linear control the situation is totally opposite.
\bigskip

\section{Linear Control} 
Let consider an open-loop vector nonlinear discrete system
\begin{equation} \label{1}
x_{n+1} =F_h\left(x_{n} \right),\, \, x_{n} \in\mathbb  R^{m} ,\, \, n=1,\, \, 2,\, \, \ldots \, \, ,
\end{equation}
which has an unstable equilibrium $x^{*} $, perhaps more then one. It is assumed that a differentiable function $F_h$ depends on finitely-many parameters, and that for each parameter vector $h$ from the admissible set of these parameters $H$ it is defined on some bounded simply-connected set of $m$-dimensional space and maps it in itself. A location of the equilibrium of $x^{*} $ and the spectrum
$\left\{\mu _{1} ,\, \ldots \, ,\, \mu _{m} \, \right\}$ of Jacobi matrix  $F'_h(x^{*} )$  depend on these parameters, i.e. the multipliers
$\mu _{j} ,\, \, j=1,\, \ldots \, ,\, m$ in general are known only approximately. If the admissible set  $H$ consists only on one point then the function $F_h$ and the multipliers are known precisely. 

In other words,  instead of the functions family $F_h$ the set of possible location of the multipliers $M\subset \bar C$ is considered. Here $ \bar C$ denotes the extended complex plane. For example, if $F_h(x)=hx(1-x), h\in (1,4]$ then $M=[-2,1).$
If $F_h(x)=h\sin \pi x, h\in [-1,1/\pi)$ then $M=[-\pi,1).$ If $F_h(x,y,z)=h(\sin(x+y),\sin(y+z),\sin(z+x)), h\in [-2,-1]$ then
$$
M=\left\{\rho e^{i\frac\pi3}:\rho\in[1,2]\right\} \cup\left\{\rho e^{i\frac{2\pi}3}:\rho\in[1,2]\right\}\cup \left\{\rho e^{i\pi}:\rho\in[2,4]\right\} 
$$
If the function $F_h$ is known exactly (i.e. the set $H$ consists of one element) then the set $M$ consists of no more then $m$ points of complex plane. 

It is required to determine a necessary condition on the set $M,$ that allows local stabilization of the equilibrium  $x^*$ of the system \eqref{1} for all admissible parameters by {\it one} additive control of the form
\begin{equation} \label{2}
u=-\sum _{j=1}^{N-1}\varepsilon _{j} \left(x_{n-j} -x_{n-j+1} \right) ,
\end{equation}
i.e. for all $\mu _{j} \in M, j=1,\, \ldots \, ,\, m$.

A necessary conditions will be stated in terms of the size of the set $M$ and its connected components. At the same time the problem of determining general conditions on the set $M$ that guaranty robust stability still open.

Let underline some important properties of the considering control. The control depends not on a state of the system but on the difference of the states in certain pervious instances of time. At synchronized state $x_{n} =x_{n-1} $ the control \eqref{2} became zero, i.e. the close-loop system takes the same form as it is with no control. It means,
that the equilibriums  of open and closed-loop systems coincide and the control \eqref{2}
does not depend on a position of the equilibrium which is unknown.

The characteristic polynomial for the linear part of the closed system\eqref{1} and \eqref{2} is
\begin{equation} \label{3}
f(\lambda )=\prod _{j=1}^{m}\left(\lambda ^{N} +\left(-\mu _{j} +a_{1} \right)\lambda ^{N-1} +a_{2} \lambda ^{N-2} ...+a_{N} \right) ,
\end{equation}
where $a_{1} =-\varepsilon _{1} ,\, \, a_{j} =\varepsilon _{j-1} -\varepsilon _{j} ,\, \, j=2,\, \, \ldots \, \, ,\, \, N-1,\, \, a_{N} =\varepsilon _{N-1} \, \, $. It is clear that $\sum _{j=1}^{N}a_{j} =0 $. The multipliers $\{\mu_1,\dots\mu_m\}$ depends on the parameter vector $h.$

Denote
$f(\lambda )=\prod _{j=1}^{m}\chi _{\mu _{j} } (\lambda ) $, where
\begin{equation} \label{4}
\chi _{\mu } (\lambda )=\lambda ^{N} +\left(-\mu +a_{1} \right)\lambda^{N-1} + a_{2} \lambda ^{N-2} ...+a_{N} .
\end{equation}

Assume that the $M$ is not empty, i.e.  that for some $\mu =\mu _{0}\in M $ the polynomial \eqref{4}  is  Schur stable. Since $\chi_{\mu_0}(1)=1-\mu_0$ and $\chi_{\mu_0}(\lambda)>0$ for large values of  $\lambda$ then  $\chi_{\mu_0}(1)>0$ otherwise
by a mean value theorem there is a root outside a unit disc which is impossible. Thus  $\mu _{0} <1.$
Vieta theorem implies that the sum of the coefficients of Schur stable polynomial does not exceed $2^N$, i.e. $1-\mu _{0} <2^N.$ On the other hand, if $0<1-\mu _{0} <2^{N} $, then there exist coefficients
$a_{1} ,...,a_{N}, \sum _{j=1}^{N}a_{j} =0,$
such that the polynomial  \eqref{4} is Schur stable at $\mu =\mu _{0} $  \cite{[11]}. Thus, all real numbers from the set $M$ are in the interval $(-2^N+1,1)$.

Note that zeros of polynomials continuously depend on parameters and with the change of $\mu $ 
can escape from the disc. In this case the sequence of bifurcations is observing 
in the system \eqref{1} closed by the control \eqref{2}, which under quite general assumptions lead to emergence of a chaotic attractor. The first bifurcation value of the parameter corresponds to loss of the stable equilibrium by the system. We will assume that local stabilization of an equilibrium means regularization of a chaotic behavior of system solutions up to complete suppressing the chaos in the system. It happen if  basin of attraction coincides with the whole space of the initial values.

Let fixed $\mu_0\in M$ define the region
\[A_{N} (\mu _{0} )=\left\{\left(a_{1} ,...,a_{N} \right):\, \lambda ^{N} +(a_{1} -\mu _{0} )\lambda ^{N-1} +a_{2} \lambda ^{N-2} ...+a_{N} \;\mbox{is Schur stable}
\right\}
\]
For a vector of coefficients $a=\left(a_{1} ,...,a_{N} \right)$ from a region $A_{N} (\mu _{0} )$ define the set
\[M_{a} =\left\{\mu \in \mathbb C:\, \lambda ^{N} +(-\mu +a_{1} -\mu _{0} )\lambda ^{N-1} +a_{2} \lambda ^{N-2} ...+a_{N} \; \mbox{is Schur stable}
\right\} \]
which contains $\mu_0$ and therefore is non-empty. The set $M_a$ is not necessary a connected \cite{[12]}. 

The existence of the control \eqref{2} that locally stabilize the equilibrium of the system \eqref{1} for all admissible values of parameters means existence of a vector $a\in A_N(\mu_0)$ such that $M\subset M_a.$

Let
$$
M_a=\bigcup_{j=1}^k M_a^{(j)},
$$
where $k$ is a number of simply-connected components $M_a^{(j)}.$ Define diameters of the sets $M_{a}$ and $M_{a}^{(j)}, j=1,...,k,$ i.e. the quantities 
$$
d(M_a)=\sup_{z_1\in M_a, z_2 \in M_a}|z_1-z_2|
$$
and
$$
d(M_{a}^{(j)})=\sup_{z_1\in M_{a}^{(j)}, z_2 \in M_{a}^{(j)}} |z_1-z_2|,\;\;  1\le j\le k.
$$ 

If the diameter of the set $M$ will be bigger then $\max_{a\in A_N(\mu_0)}\{d(M_a)\}$ or a diameter of some connected 
component of the set $M$ will be bigger then  $\max_{a\in A_N(\mu_0),1\le j\le k}\{d(M_{a}^{(j)})\}$ then there does not exist  the control \eqref{2} that locally stabilizes the equilibrium of the system \eqref{1} for all admissible values of parameters incorporated in that system.

Let us turn to the evaluation of the diameters of the sets 
$M_{a}$ and $M_{a}^{(j)}.$ \bigskip

\subsection{ Preliminary results.}

From the statement of the problem follows that the coefficients $a_{1} ,...,a_{N} $ are real, and
$$
\sum _{j=1}^{N}a_{j} =0.
$$ 
In this section the above restrictions wont be used.

So, let $a_{1} ,...,a_{N} $ be arbitrary {\it complex} numbers from $A_{N} (\mu _{0} )$.
Let write the polynomial $\chi _{\Delta\mu +\mu _{0} } (\lambda )$ in the form
\[\lambda ^{N} +(-\Delta\mu +a_{1} -\mu _{0} )\lambda ^{N-1} +a_{2} \lambda ^{N-2} ...+a_{N} =\lambda ^{N} -\Delta\mu \lambda ^{N-1} +p(\lambda )\]
and denote $q(z)=(a_{1} -\mu _{0} )z+\, \ldots \, \, +a_{N} z^{N} $,
$\Phi (z)=\frac{z}{1+q(z)} $.

The following lemma formalizes a very useful and very practical observation due to A.Solyanik \cite{S}. 

\begin{lemma}\label{l1}
{ Polynomial $\chi _{\mu +\mu _{0} } (\lambda )=\lambda ^{N} -\Delta\mu \lambda ^{N-1} +p(\lambda )$ is Schur stable if and only if
\begin{equation} \label{5}
\frac{1}{\Delta\mu } \in \bar{\mathbb C}\backslash \Phi (\overline{\mathbb D}),
\end{equation}
where $\mathbb D = \{z\in\mathbb C: |z|<1\},$ $\overline{\mathbb D} = \{z\in\mathbb C: |z|\le1\}.$ }
\end{lemma} 

{\bf Proof.} Polynomial $\chi _{\Delta\mu +\mu _{0} } (\lambda )$ is Schur stable if and only if the image of the set $\bar{\mathbb C} \backslash \mathbb D$ under the map $\chi _{\Delta\mu +\mu _{0} } (\lambda )$ does not contain zero, i.e. $\chi _{\Delta\mu +\mu _{0} } (\lambda )\ne 0$ for all $\lambda \in\bar{\mathbb C}\backslash \mathbb D$.

This is equivalent to
$$
\frac{1}{\Delta\mu } \ne \frac{\frac{1}{\lambda } }{1+\frac{1}{\lambda ^{N} } p(\lambda )}, \; \lambda \in\bar{\mathbb C}\backslash \mathbb D
$$
or
$$
\frac{1}{\Delta\mu } \ne \frac{z}{1+q(z)}, \;z \in \bar{\mathbb D}
$$
Thus, a polynomial $\chi _{\Delta\mu +\mu _{0} } (\lambda )$ is Schur stable if and only if $\frac{1}{\Delta\mu } \notin \Phi (\overline{\mathbb D})$
or $\frac{1}{\Delta\mu } \in \bar{\mathbb C}\backslash \Phi (\overline{\mathbb D})$. The Lemma is proved.\\

Let introduce an inversion of complex numbers by the rule $\left(z\right)^{*} =\frac{1}{\bar{z}}.$ By inverse set
we will understand the set consisting of inverse elements of the initial one. Condition of robust stability \eqref{5} is equivalent to the inclusion
\[\Delta\mu \in \left(\bar{\mathbb C}\backslash {\Phi_1 }(\overline{\mathbb D})\right)^{*} ,\]
where ${\Phi_1 }(z)=\frac{z}{1+{q_1}(z)} $ and ${q_1}(z)=(\bar{a}_{1} -\bar\mu _{0} )z+\, \ldots \, \, +\bar{a}_{N} z^{N} $.

Since the polynomial $\lambda ^{N} +p(\lambda )$ is Schur stable, all poles of the function $\Phi (z)$ lie outside the unit disc $\mathbb D$. I.e. the 
function $\Phi (z)$ is analytic in $\mathbb D$. The size of the set associated with $M_{a} $ are related to the
size of $\Phi (\mathbb D)$.  The image set $\Phi (\mathbb D)$ is an open but not necessarily a simply connected set. Denote by $\Phi ^{s} (\mathbb D)$
a minimal simply connected set containing $\Phi (\mathbb D)$. \\

\begin{lemma}\label{l2} The set $\Phi ^{s} (\mathbb D)$ contains the disc of radius 1/4.
\end{lemma}

{\bf Proof}. By Riemann mapping theorem there exists a function $\varphi (z)$ univalent in $\mathbb D$, such that $\varphi (\mathbb D)=\Phi ^{s} (\mathbb D)$, wherein $\varphi (0)=0$, $\varphi '(0)>0$. Consequently, $\varphi ^{-1} (\Phi (\mathbb D))\subseteq \mathbb D$. This means that
function $F(z)=\varphi ^{-1} (\Phi (z))$ posesses the inequality $\left|F(z)\right|<1$ for $z\in \mathbb D$ and satisfies the conditions of Schwarz's lemma, from where
$|F'(0)|<1$. But $F'(0)=\Phi '(0)\left(\varphi ^{-1} \right)^{{'} } _{z=0} =\left(\varphi ^{-1} \right)^{{'} } _{z=0} $. This gives the required estimates $\left(\varphi ^{-1} \right)^{{'} } _{z=0} <1$, $\varphi '(0)>1$. K\"obe theorem \cite{[12]} implies that the set $\varphi (\mathbb D)$ contains a circle of radius
$\frac{\varphi '(0)}{4} $ and therefore one of the radius 1/4.
The lemma is proved.\\

{\bf Remark.} Lemma 2 can be viwed as an extension of K\"obe theorem to
arbitrary mappings of the unit disc, not necessarily univalent:
a minimal connected domain containing the image of the unit disc
under {\it any} analytic in the disc function of the type
\[z+c_{1} z^2+c_{2} z^{3} +\, \ldots ,\]
contains a central disc of radius $\frac{1}{4} $.\\

\begin{lemma}\label{l3} 
The set $\Phi (\mathbb D) $ contains a disc of radius 1/16.
\end{lemma}

The proof follows directly from  Caratherodory theorem \cite{[14]}: if analytic in the disc $\mathbb D$ function $c_0z+c_1z^2+c_2z^3+...$ does not have zeros in $\mathbb D\backslash\{0\}$
then for any exceptional value $\gamma$ in $\mathbb D$ we have $|\gamma|\ge\frac{|c_0|}{16},$ i.e. if $\gamma \not \in \Phi (\mathbb D) $ then $\gamma\ge\frac1{16}.$
\bigskip

\subsection{Main result.}

\begin{theorem}\label{t1}
Let for some $\mu_0$ the set $ A_{N}(\mu_0)$ is not empty and let $a=(a_1,\dots a_N)\in A_N(\mu_0).$ Then $d(M_a)\le 16.$
\end{theorem}

\begin{proof} Lemma 2 implies that the set $\bar C\backslash {\Phi }(\overline{\mathbb D})$ does not contain the central  disc of radius 1/16.  
Then Lemma 1 implies that for $\mu_1\in M_a$  
$$
\frac 1{|\Delta \mu|}=\frac 1{|\mu_0- \mu_1|}\ge \frac1{16},
$$
therefore $|\mu_0- \mu_1|\le16.$ This estimate is valid for arbitrary $\mu_0$ and $\mu_1$ from $M_a$. The theorem is proved.
\end{proof}

\begin{theorem}\label{t2}
 Let for some $\mu_0$ the set $ A_{N}(\mu_0)$ is not empty and let $a=(a_1,\dots a_N)\in A_N(\mu_0).$ Then $d(M^{(j)}_a)< 4,$ where $M^{(j)}_a$ is a connected component of the set $M_a$  that contains $\mu_0.$ 
\end{theorem}

\begin{proof} Lemma \ref{l2} implies that that the set $\bar{\mathbb C}\backslash {\Phi^s }(\overline{\mathbb D}$ does not contain a central disc of radius 1/4. Now, by Lemma \ref{l1} if $1/\mu_1\in \bar{\mathbb C}\backslash {\Phi^s }(\overline{\mathbb D}$ then $\mu_1\in M_a^{(j)}.$

From here
$
\frac 1{|\Delta \mu|}=\frac 1{|\mu_0- \mu_1|}>\frac1{4},
$
and $|\Delta\mu|=|\mu_0- \mu_1|<4.$ The estimate is valid for any  $\mu_0$ and $\mu_1$ in  $M^{(j)}_a$ . The theorem is proved. 
\end{proof}

Remark that the value of the radius in the Theorem 2 can not be reduced in general.
Indeed, for $\mu _{0} =0$, $\varepsilon \in \left(0,\, 1\right)$ the vector $\left(a_{1} ,\, a_{2} \right)=\left(2(1-\varepsilon ),\, 1-\varepsilon \right)$ belongs to the set $A_{2} (0)=\{(a_1,a_2): a_2+1>|a_1|, a_2<1 \}.$ Then the multiplier
$\mu_1=4-3\varepsilon$ belongs to the set $M_a.$ It is clear that $\sup_{\varepsilon\in (0,1)}|\mu_0-\mu_1|=4.$

On the Fig. 1 the set $\Phi(\mathbb D)$ is displayed for $a_1=2(1-\epsilon),$
$a_2=1-\epsilon$ where $\epsilon=0.1.$ It is simply connected and contains entirely the central disc of the radius 1/4 (a black spot at the origin). The set $M_a$ is inverse with respect to the unit circle of the exterior of the set $\Phi(\mathbb D)$. It is entirely containing in the disc of radius 4, which is an inversion of the exterior of the disc of radius 1/4 (see Fig. 2). With decreasing $\epsilon$  the diameter of the set $M_a$ approaches 4.

\centerline{
\includegraphics[scale=0.35]{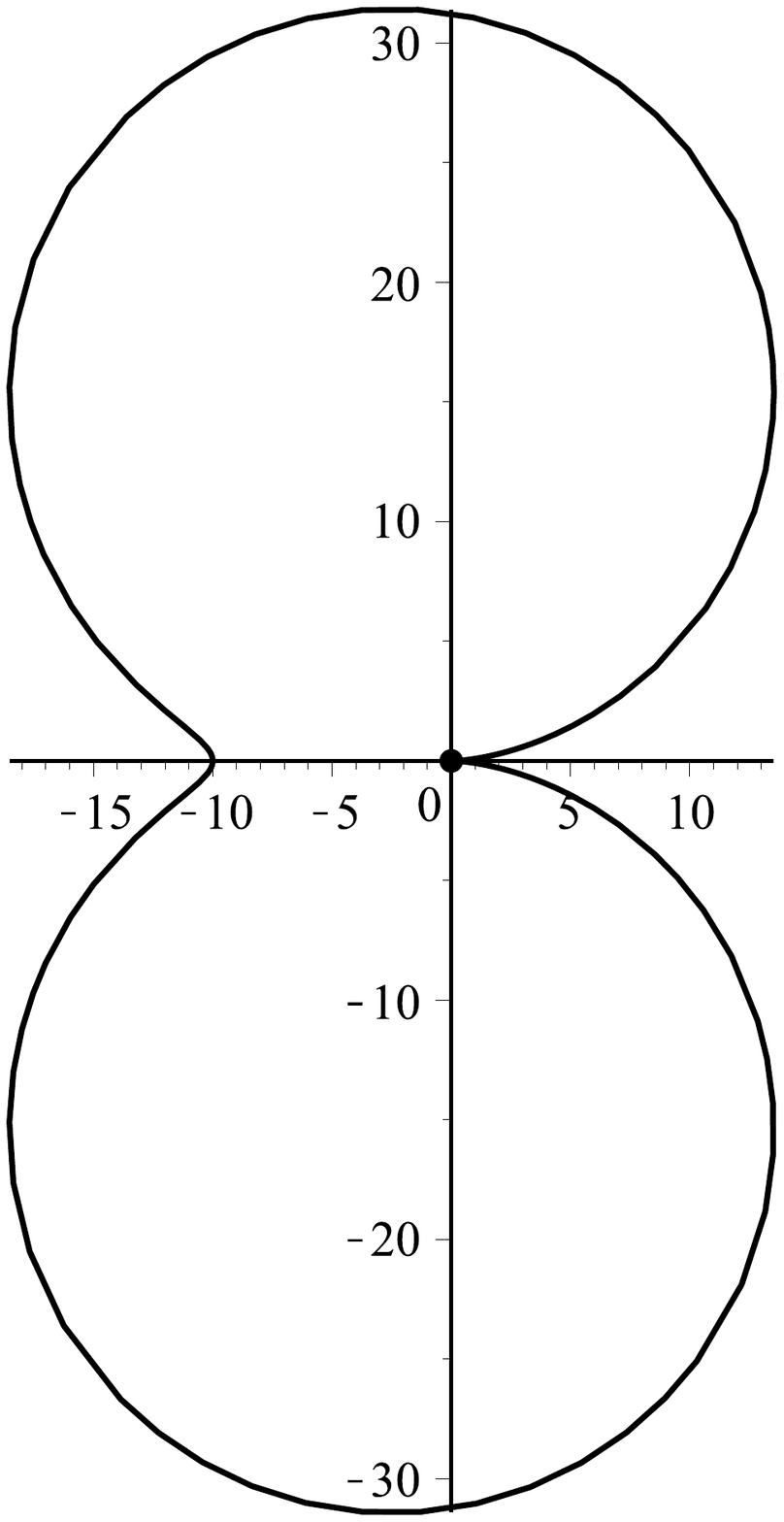}
}
\centerline{Fig. 1}


Figure 1 
displays image of the circle $\mathbb D$ under the mapping $F(z)=\frac{z}{1+a_{1} z+a_{2} z^{2} } $ and center circle
radius 1/4 ($\varepsilon =0.1$).

\centerline{
\includegraphics[scale=0.25]{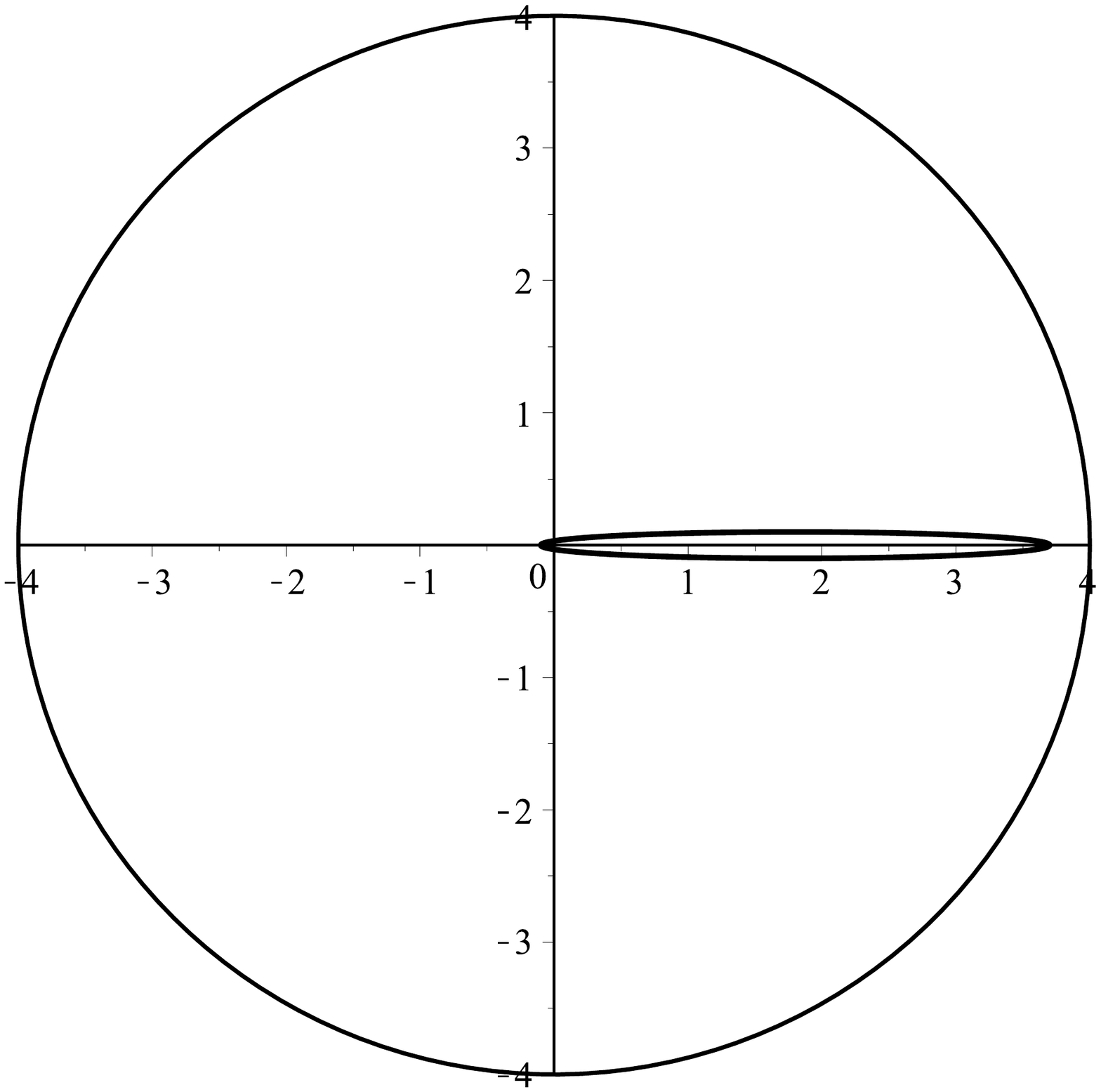}
}
\vspace{1cm}
\centerline{ Fig. 2}

Fig. 2 
displays the region $M_{a} $ and a central circle of radius 4 ($\varepsilon =0.1$).

\begin{theorem}\label{t3}
If diameter of the set $M$ is larger then 16 or a diameter of any of its connected component is larger then 4, then for any $N$ there is no control (2) that stabilizes equilibrium of the system (1) for all admissible parameters of the system.
\end{theorem}

\begin{proof}
If exist stabilizing control (2) then there exists a vector of coefficients $a=(a_1,\dots,a_N)$ such that the family of polynomials
$$
\{
\lambda^N-\mu\lambda^{N-1}+a_1\lambda^{N-1}+a_2\lambda^{N-2}+\dots+a_N:\mu\in M
\}
$$ 
is Schur robust stable. By theorem \ref{1} the diameter of the set $M$ cannot exceed 16 while the diameter of any of its connected component cannot exceed 4. The theorem is proved.
\end{proof}

Now, let us consider the case where the function $F_h$ in the system (1) is defined precisely, i.e. the set of the admissible parameters consists of one point: $H=\{h_0\}$. Denote $F_{h_0}=F$ and $x^*_{h_0}=x^*.$

\begin{theorem}\label{t4}
If  spectrum $\{\mu_1,\dots,\mu_m \}$ of Jacobi matrix $F^\prime(x^*)$ has a diameter greater then 16, then there is no control \eqref{2} that stabilizes the equilibirum $x^*.$ 
\end{theorem}

If $M$ consists of real numbers then it is possible to state not only a necessary condition of existing the stabilizing control (Theorem \ref{t5}) but a sufficient as well (Theorem \ref{t6}).

\begin{theorem}\label{t5}
Assume that the spectrum of Jacobi matrix of the system (1) is real for all admissible values of the parameters. And assume that the set $M$ is simply connected. Then a necessary condition for the existence of stabilizing control (2) is the length of the interval, defined by $M$ to be at most 4 and all the numbers from it are less then 1.
\end{theorem}
  
\begin{theorem}\label{t6}
Let $M=(a,b),$ where $-3<a<b<1.$ Then there is a control $u_n=-\epsilon(x_{n-1}-x_n)$ that locally stabilizes the equilibrium of the system (1).
\end{theorem}

\begin{proof}
The characteristic polynomial of the closed-loop system can be written in the form
$$
f(\lambda)=\prod_{j=1}^m(\lambda^2+(-\mu_j-\epsilon)\lambda+\epsilon),
$$
where $\mu_j\in (a,b).$ If $\epsilon\in (-\frac{1+a}2,1)$ the vectors $(-\mu_j-\epsilon,\epsilon), j=1,...,m$ belongs to the stability region 
$A_2(0)=\{(a_1,a_2):a_2+1>|a_1|,a_2<1\}.$ Therefore for $\epsilon\in (-\frac{1+a}2,1)$ the characteristic polynomial of the closed-loop system is Schur stable for all $\mu_j\in (a,b).$ The theorem is proved.
\end{proof}

Let us mention, that even when all conjectures of the Theorem \ref{t6} are fulfilled, i.e. when it is known that there exists linear stabilizing control, practically it cannot be implemented. As it mentioned in \cite{[6]} the basin of attraction of the stable equilibrium can turns to be too small.

\subsection{Examples}

{\bf A.} Consider the controllable chaotic system 
\begin{equation} \label{6}
x_{n+1} =h\sin(\pi x_{n})+u_{n} ,
\end{equation}
\begin{equation} \label{7}
u=-\sum _{j=1}^{N-1}\varepsilon _{j} \left(x_{n-j} -x_{n-j+1} \right),
\end{equation}
where $h\in(-h_0,1/\pi), 1/\pi<h_0\le1.$ In this case $M=(-\pi h_0,1).$

If $h_0\le\frac3\pi$ then accordingly to the Theorem 6 there exists a control 
\eqref{6} for $N=2$ that locally stabilizes the trivial equilibrium $x^*=0$ of the system \eqref{6} for all $h$ from $M.$ As closer $h_0$ is to $3/\pi$ as smaller become the basin of attraction of the trivial equilibrium. 

Now, let $3/\pi<h_0\le1.$ For every fixed  $h\in(-h_0,-\frac{3}{\pi})$ there exists the stabilizing control \eqref{7}. I.e. the strength coefficients in
\eqref{7} should depend on $h.$ At the same time $\mu\in(-\pi,-3)$ and the condition $2^2<1-\mu\le2^3$ implies that for $N=3$ it is possible the stabilization with two-steps control $u_n=-\epsilon_1(h)(x_{n-1}-x_{n})-\epsilon_2(h)(x_{n-2}-x_{n-1}).$

However, there is no $N$ that admits the control \eqref{7} independent on $h$ and stabilizing the trivial equilibrium for all $h\in(-h_0,1).$
\bigskip

\textbf{B.} Let consider controllable chaotic system linearised around an equilibrium:
\begin{equation} \label{8}
\begin{array}{c} {x_{n+1} = \mu_1 x_{n}+u_{n}^{(1)}} \\ 
{y_{n+1} =\mu_2 x_{n}+u_{n}^{(2)} } \end{array},
\end{equation}
\begin{equation} \label{9}
u_{n} =
\left(\begin{array}{c} {u_{n}^{(1)} } \\ {u_{n}^{(2)} } \end{array}\right)
=
-\sum _{j=1}^{N-1}\varepsilon _{j} 
\left(
\begin{array}{c} {x_{n-j} -x_{n-j+1} } \\ {y_{n-j} -y_{n-j+1} } \end{array}
\right),
\end{equation}
where $\mu_j\in M,j=1,2.$

Let us show that there exists a set $M$ of the diameter larger then 4 that admits a local stabilization of the trivial equilibrium of the system \eqref{8} by the control \eqref{9} for all $\mu_j\in M,j=1,2.$.

To construct such control we will use the results from \cite{[16]}. Let $\mu_1^0=-\frac{79}{24}$ and $\mu_2^0=-\frac{23}{24}.$ Since
  $\mu_1^0+\mu_2^0=\frac{17}4>4$ then if the 
stabilizing control does exists then the set $M$ cannot be connected. Let determine the necessary restriction on $N.$  Since $1-\mu^0_1\approx 4.29<2^3$ 
then $N\ge3.$ If $N=3$ then the polynomials $f_j(\lambda)=\lambda^3-\mu_j\lambda^2+a_1\lambda^2+a_2\lambda+a_3,\; j=1,2$ should be stable. 

For any stable polynomials $f_j(\lambda)$ of the degree $N\le 3$ the polynomials from the family $\{\theta f_1(\lambda)+(1-\theta)f_1(\lambda)\}$  should be stable as well \cite{[16]} . However by the 
the Theorem \ref{t3} it is impossible. Therefore, $N\ge 4.$ Let consider the vector $a=(-\frac76,\frac32,0,-\frac13)\in A_4\left(\frac{23}{24}\right)$ and let construct the image of the of the unit disc under the map $\Phi(z)=\frac z{1-\frac{23}{24}z-\frac76z+\frac32z^2-\frac13z^4}$ (see the Fig. 3). This image is not a simply connected set 
 (see the Fig. 4). The figures 5 and 6   
 displays the region of the possible location of $\Delta\mu$ of the displacement of the multiplier $\mu$ from $\mu_2^0.$


\centerline{
\noindent \includegraphics[scale=0.35]
{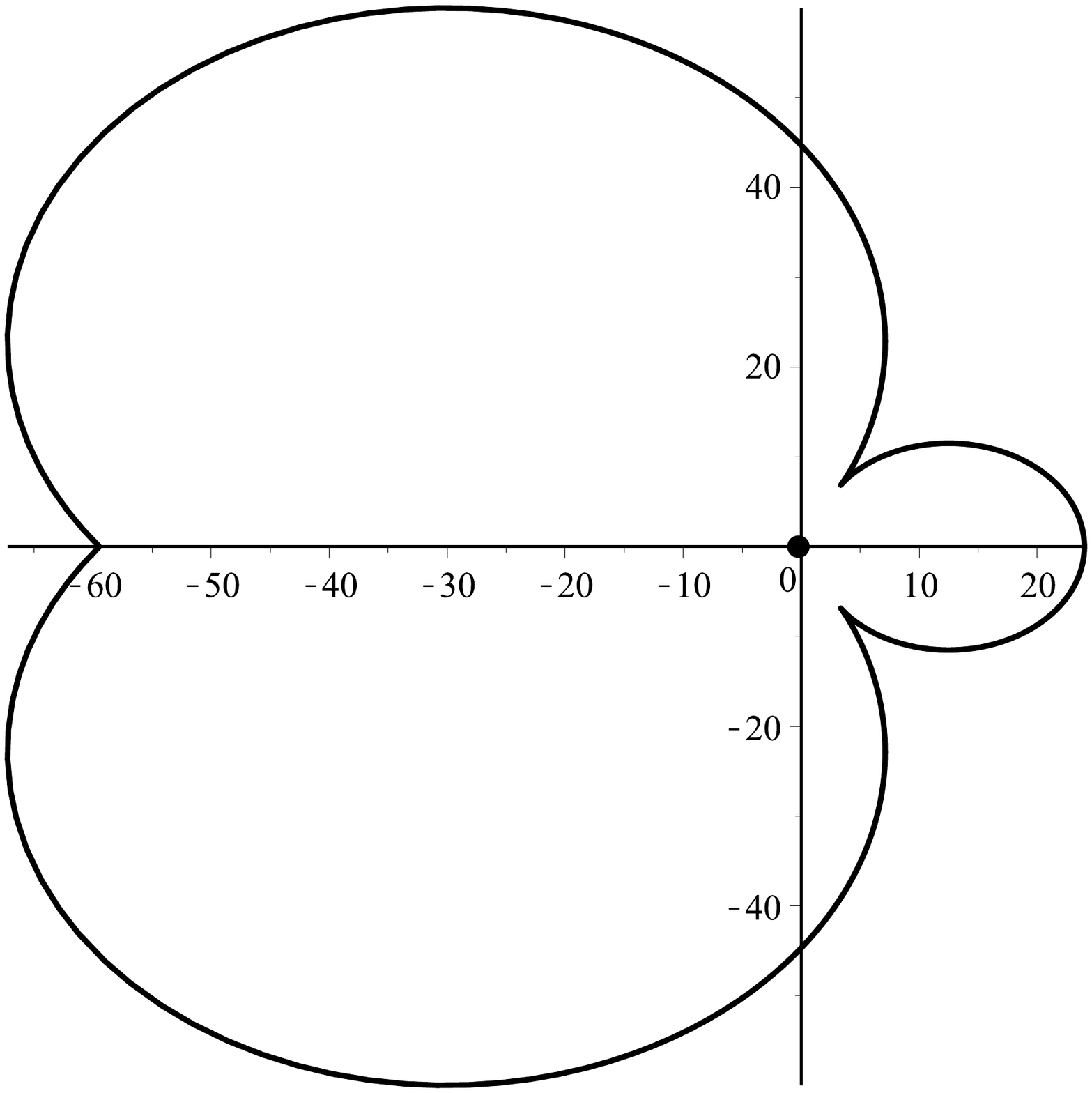}}

\vspace{1cm}
\centerline{Fig. 3. The image of the unit disc $\Phi(\mathbb D)$.}

\vspace{-1cm}
\centerline{ \includegraphics[scale=0.35]{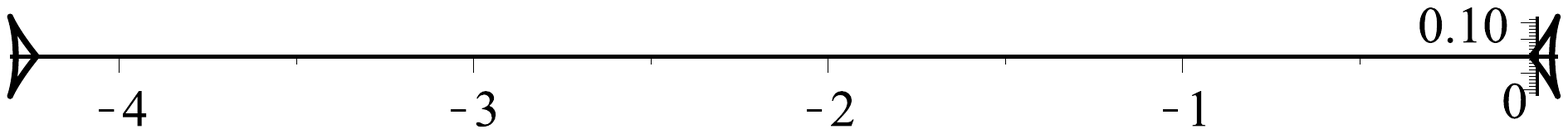}}
\vspace{-4.5cm}
{Fig. 4 The set $ \left(\bar C\backslash \bar{\Phi }(\overline{\mathbb D})\right)^{*} $ is not connected and can be written as a union of two simply-connected sets}\\

\centerline{
\includegraphics[scale=0.2]{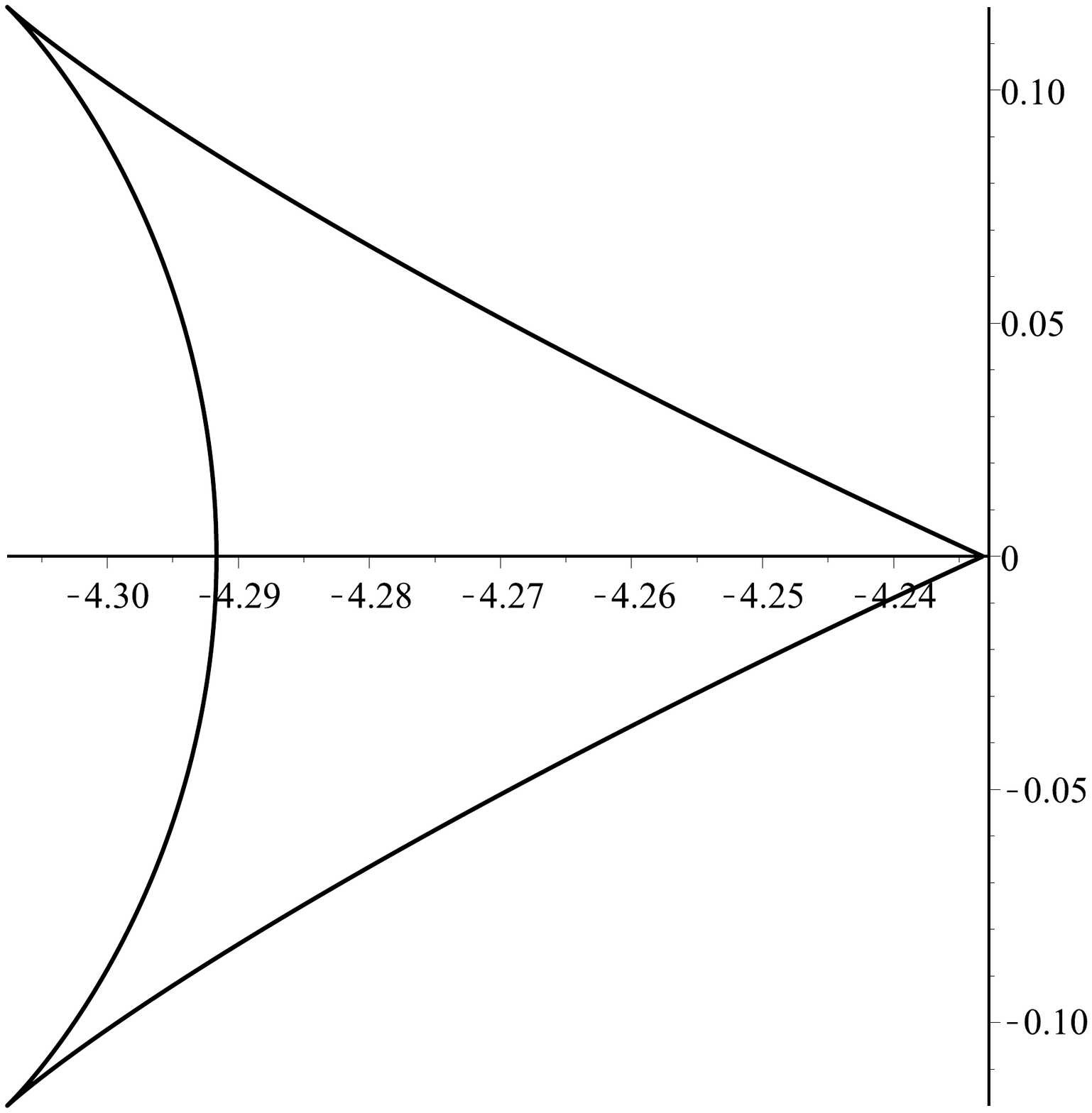}\hspace{2cm} \includegraphics[scale=0.2]{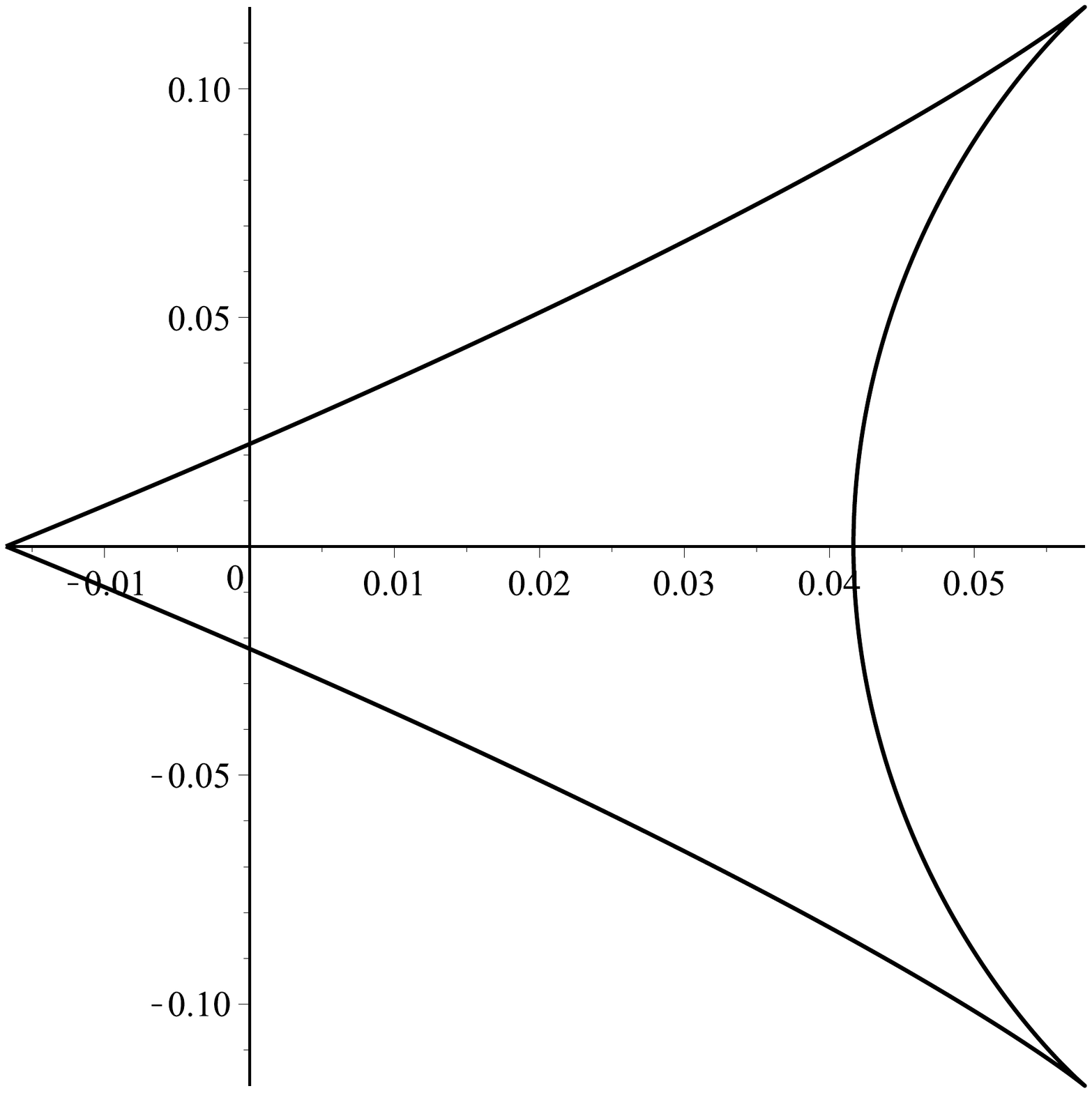}
}
\centerline{ Fig.5 \hspace{5cm} Fig.6}
\medskip

For the existence of the stabilizing control the equality $\mu_1^0=\mu_2^0+\Delta\mu$ should happen for some 
$\Delta\mu \in \left(\bar C\backslash \bar{\Phi }(\overline{\mathbb D})\right)^{*}.$ Indeed, the equality is valid for $\Delta\mu= -\frac{17}{4}$. Therefore for
$M=\{\mu: \mu-\mu_2^0\in M_1\}$ the stabilization is possible, the strength coefficients of the stabilizing control \eqref{9} are defined by the vector $a$ and they are
$\varepsilon_1=\frac78,$ $\varepsilon_2=-\frac13$
and $\varepsilon_3=-\frac13.$\\

{\bf C.} Let us consider a system
\begin{equation}\label{10}
\left( 
\begin{array}{c}
x_{n+1} \\
y_{n+1} \\
z_{n+1} \\
\end{array} 
\right)=h
\left( 
\begin{array}{c}
\sin(x_{n}+y_{n})\\
\sin(y_{n}+z_{n})\\
\sin(z_{n}+x_{n})\\
\end{array} 
\right),\qquad h\in H
\end{equation}
The system \eqref{10} has a trivial equilibrium that corresponds to the set of multipliers
$$
M_1=\left\{|h|e^{i\frac\pi3}, |h|e^{-i\frac\pi3}, 
-2|h|\right\}
$$
If $H=\{-h_0\}$ then for $|h_0|>\frac{16}{\sqrt 3}$ diameter of the set $M_1$ is bigger then 16 and by the Theorem \ref{t3} there is no control \eqref{2} that stabilizes the trivial equilibrium of the system \eqref{10} for $h=-h_0.$

Let $H=(-h_0,0).$ If $h_0>2$  there is no control \eqref{2} that stabilizes the trivial equilibrium of the system \eqref{10} for all $h\in (-h_0,0).$

\subsection{Conclusion}
In the above section the properties of the control with  specific structure that involves difference of the system stages computed in a certain shifted instants of time. Such type controls as well as classical ones, have a long history of applications in problems of stabilization (detecting) of unknown equilibriums or cycles in the systems with continuous or discrete time. However, the possibility of application of such control (or definite impossibility) has studied only in simplest partial cases. 

The action of such controls can be explained by the following aspects:\\

a) The space of the initial stages of the system \eqref{1} has dimension $m$ while after closing the system by the control \eqref{2} the space of the initial stages is changing, its dimension became $(N-1)m.$\\

b) The space of the initial stages of the closed-loop system
is splitting in the collection of invariant sets among them
appear stable minimal sets (and the basins of their attraction).\\

c) In the initial $m$-dimensional space this minimal sets correspond locally stable cycles, therefore the chaotic structure of the solutions is regularizing.\\

By this reason the problem of local stabilization of the equilibriums of the system \eqref{1} tunes out to related to the problem of chaos suppressing.

We do not touch the problem of estimating the basins of attraction. Our main goal is to demonstrate the limitation of the applicability of the linear controls thus justify the necessity of the nonlinear control. The theorem \ref{t3} demonstrates that the set $M$ of the possible location of multipliers of the system \eqref{1} cannot be arbitrary large for any {\it linear} control \eqref{2}, i.e. its diameter cannot exceed 16 and the diameter of any connected component cannot exceed 4 regardless of the system dimension $m$ and the number $N$ in the control \eqref{2}. On contrary, the application of the {\it nonlinear} control allows to stabilize chaos in the systems with arbitrary large set of locations of multipliers \cite{[9]} by increasing the number $N$ of the strength coefficients of the control
\eqref{nlc}.

Additionally, we have risen a problem of determine necessary and sufficient conditions on  the set $M$ that guaranties stability of the system \eqref{1} by the linear control \eqref{2} in a general situation. Let us mention that this problem is not solved even in simplest cases. Say, if $M=\{\mu\},$ where $\mu$ is a real number, then the necessary and sufficient condition for the stability is $\mu<1.$ However, if $\mu$ is complex and $M=\{\mu, \bar\mu \}$ the problem is open.   
       
\section{Non-Linear Control}
Having understanding of the limited power of the linear DFC it is naturally to consider the non-linear control \eqref{nlc}. The scalar case was considered in \cite{[10]}. Below we consider the vector valued case. There is a significant difference between these cases. Namely, the scalar case deals with real
multiplier while the vector deals with complex ones.

Assume again that the multipliers are located in a region $M.$ A close-loop system is of the form 
\begin{equation} \label{dsc} 
x_{n+1}=f_h(x_n) +u_n
\end{equation} 
where the control is non-linear.

The system \eqref{dsc} linearized around an equilibrium point takes the form
\begin{equation}\label{sys}
x_{n+1}=A\cdot\sum_{j=0}^{N-1} \alpha_j x_{n-j}.
\end{equation}
where Jacoby matrix can be transformed by a non-degenerate transformation to the upper triangular form
$$
A=\left( 
\begin{array}{ccccc}
\mu_1 & a_{12}&...&... &a_{1m} \\
0 & \mu_2 & ...& ...& a_{2m} \\
... &...&...&...&...\\
0&...&\mu_j&...& a_{jm} \\
... &...&...&...&...\\
0 & ... &...&...& \mu_m 
\end{array} \right)
$$
Remarkably enough  $\alpha_0+...+\alpha_{N-1}=1$ 
therefore we bypass the first obstacle in the linear control - now the close-loop system has solutions within the range.\\

The characteristic equation for the linear system \eqref{sys} is
$$
\prod_{k=1}^m (\lambda^N-\mu_k P(\lambda))=0
$$
where 
$P(\lambda)=\sum_{j=1}^{N} \alpha_j \lambda^{N-j}$
and $\mu_k\in M\subset \mathbb C.$

Now, the problem is given a region $M\subset\mathbb C$ find a number $N$ and construct polynomials
$P(\lambda)$ such that the family of the polynomials
$
\left\{\lambda^N-\mu_k P(\lambda)\right\}
$
 is Schur stable for all $\mu_k\subset M,\; k=1,..,m.$

Modifying the expression we get
$$
\frac1{\mu_k}=\sum_{j=1}^{N} \alpha_j \lambda^{-j}=\lambda^{-N}P(\lambda)=:q\left(\frac1\lambda\right).
$$
Now, if all zeros $\lambda$ are inside the unit disc $\mathbb D$ (Shur stability) then $\frac1\lambda $ lie outside the unit disc, i.e.
we came up with the conclusion that
$$
\frac1{\mu_k}\in \left\{ \bar{\mathbb C}\backslash q(z): z\in \mathbb D\right\},\; q(1)=1,\; q(0)=0.
$$
Therefore, the admissible domain for the multipliers is
$
\left\{ \bar{\mathbb C}\backslash q(z): z\in \mathbb D\right\}^*,
$
where $E^*$ denote the inversion of  $E$ with respect to the unit circle.  \\

Then the {\it stability criteria is}  
$M\subset (\bar C \backslash q(\bar{\mathbb D})^*.$  So, we came up with the following problem of geometric complex function theory: {\it given set of multipliers $M$ find a properly normalized polynomial map  $z\to q(z)$ such that 
$M\subset \left( \bar{\mathbb C}\backslash q(\bar{\mathbb D}) \right)^*.$} Then, find the optimal or almost-optimal coefficients.\\

A crucial case of the left half plane can be resolved in the following way.  If $\mu\in\{\Re(z)<0\}\cup\{|z|<1\}$ then this domain may be consider as a union of the domains 
$M_R:=\{|z+R|<R\}\cup\{|z|<1\}.$ If $R=N/2$ then choosing the polynomial map 
$$
q:z \to \frac2N\sum_{j=1}^N(1-\frac j{N+1})z^j
$$ 
we are guaranteed 
that the image of the unit disc will be to the right of the line  $\Re(z)=1/N.$ Therefore the inverse image of all the multipliers $\mu$ with $\Re(\mu)\le1/N$ will be inside the inverse image of that vertical line which is a circle  $|z+N/2|<N/2.$ Therefore,   $M$ can be covered by $\left( \bar{\mathbb C} \backslash q(\bar{\mathbb D}) \right)^*.$  On the Fig. 8 
the left image is $q(\bar{\mathbb D})$ with $N=12$ while the right is $\left( \bar{\mathbb C}\backslash q(\bar{\mathbb D})\right)^*.$ 
 \\

\centerline{
\includegraphics[scale=0.15]{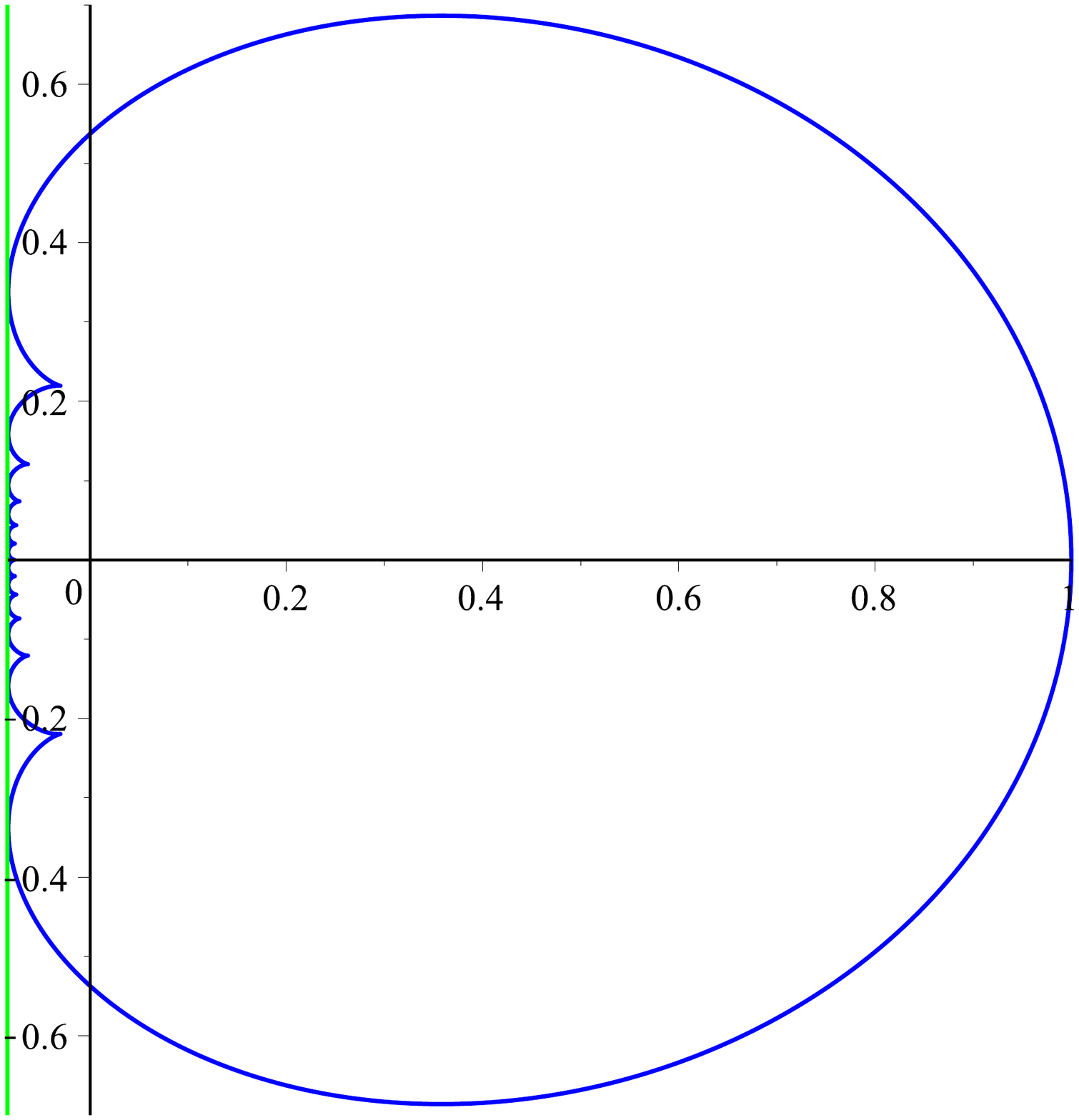}\hspace{1cm}
\includegraphics[scale=0.15]{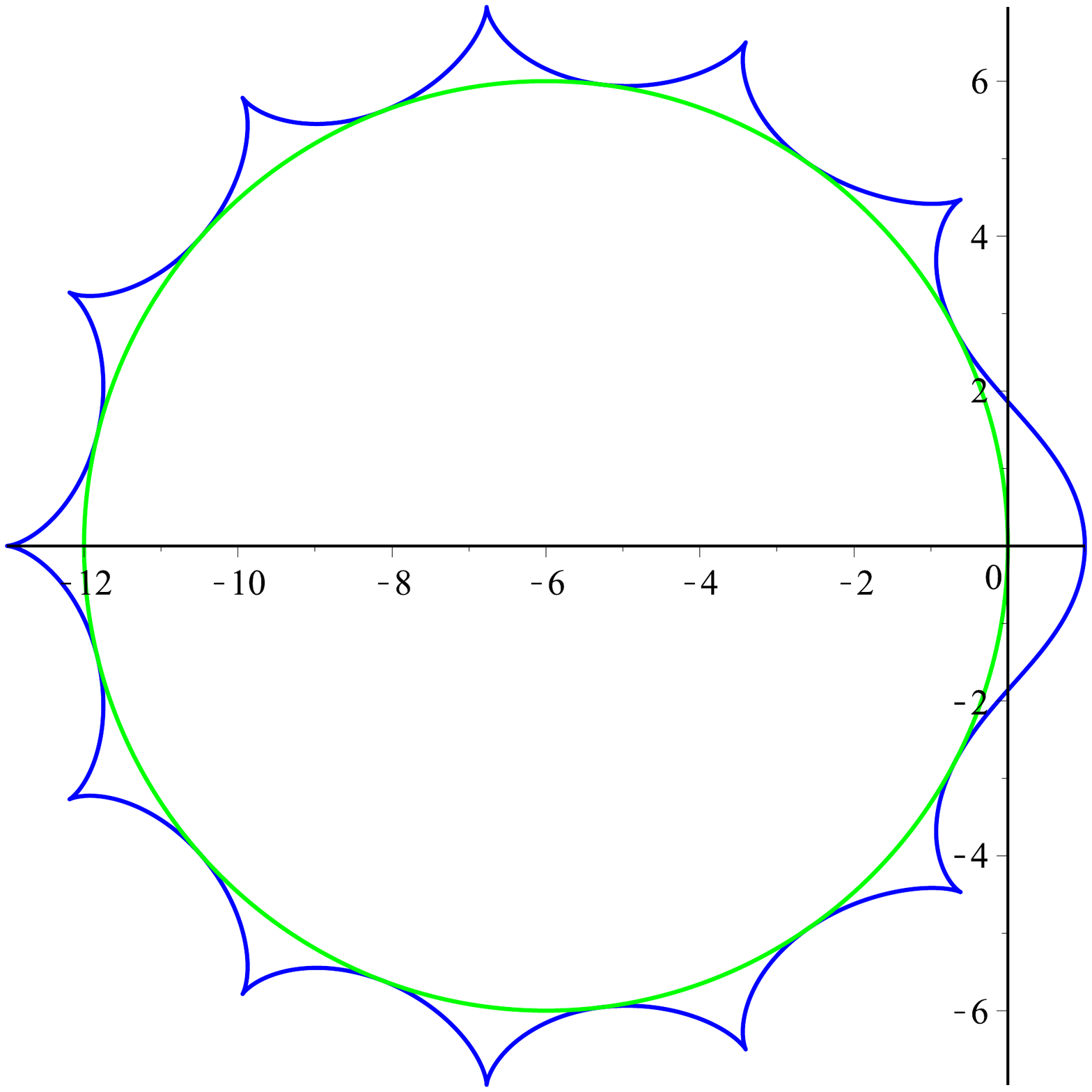}
}
\centerline{Fig. 8}


{\bf Conjecture A.} {\it The suggested  polynomial map $q$ has the lowest order among those $p:z\to p(z), p(0)=0, p(1)=1$ whose inverse image of a unit disc contains $M_R.$} \\

Now, the strength coefficients that are defined by the formulas
$$
\varepsilon_k= \frac2N\sum_{j=k+1}^N(1-\frac j{N+1}), k=1,...,N-1
$$ 
produce the non-linear control \eqref{nlc}.

As an example let us consider the Ikeda map  proposed first by Ikeda \cite{I} as a model of light going around across a non-linear optical resonator. It 2D version was considered in \cite{IDA} 
$$
x \to 1+0.9\left(x\cos\left(0.2-\frac 6{1+x^2+y^2} \right) - y\sin\left(0.2-\frac 6{1+x^2+y^2} \right) \right) 
$$
$$
y \to 0.9\left(x\sin\left(0.2-\frac 6{1+x^2+y^2} \right)  - y\cos\left(0.2-\frac 6{1+x^2+y^2} \right) \right)
$$
\bigskip

\centerline{
\includegraphics[scale=0.15]{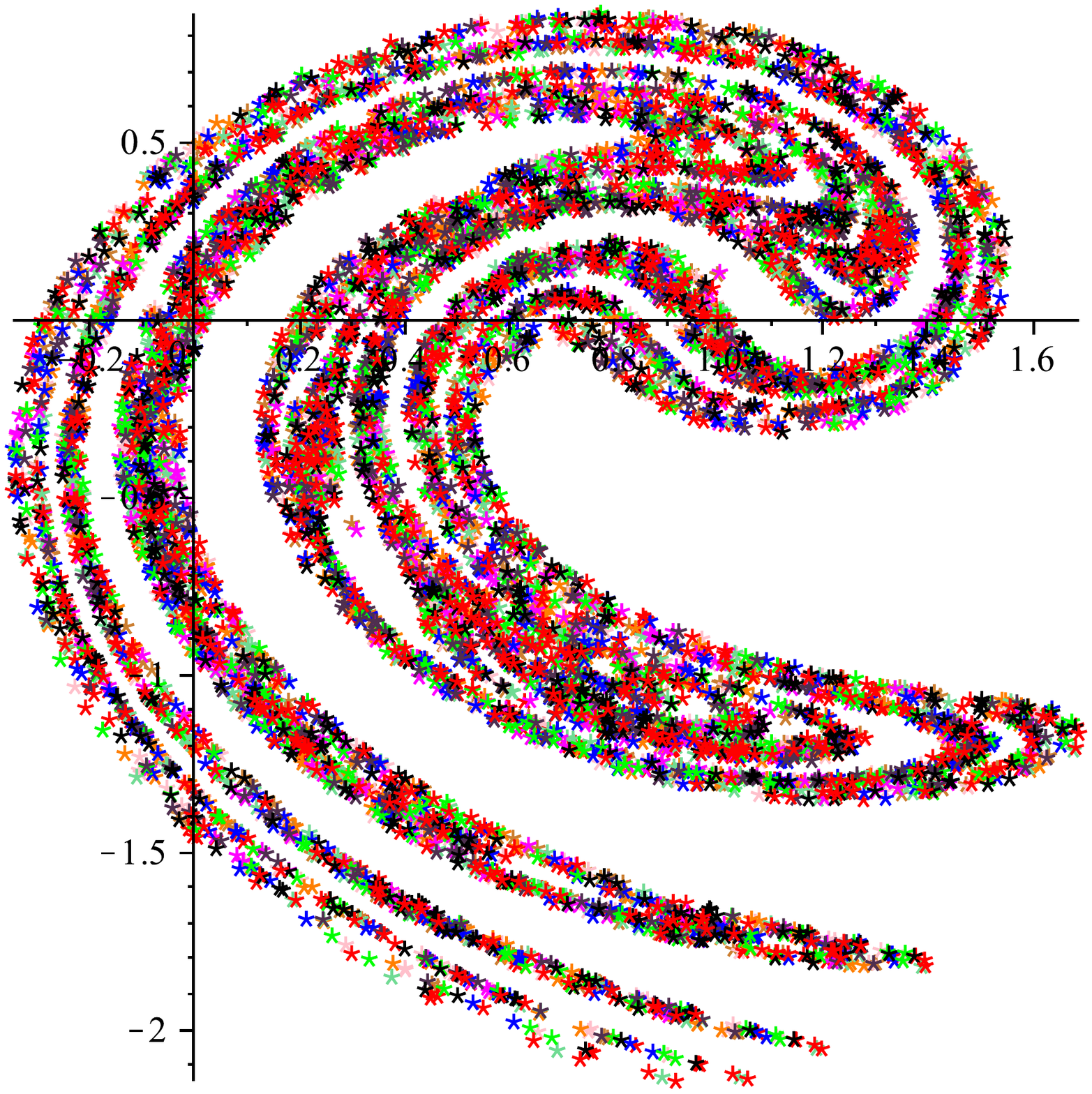}\hspace{1cm}
\includegraphics[scale=0.15]{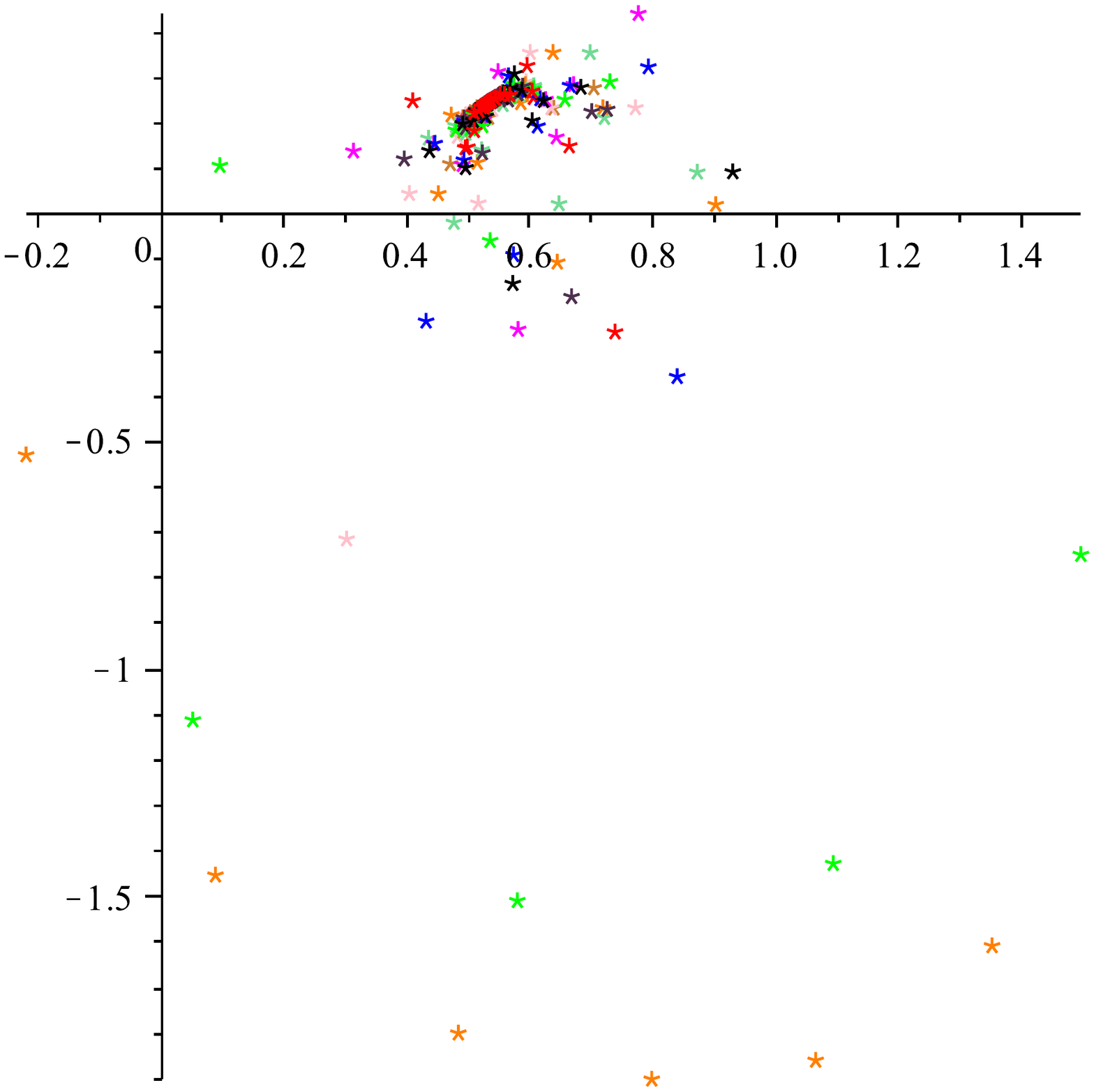}\hspace{1cm}
\includegraphics[scale=0.15]{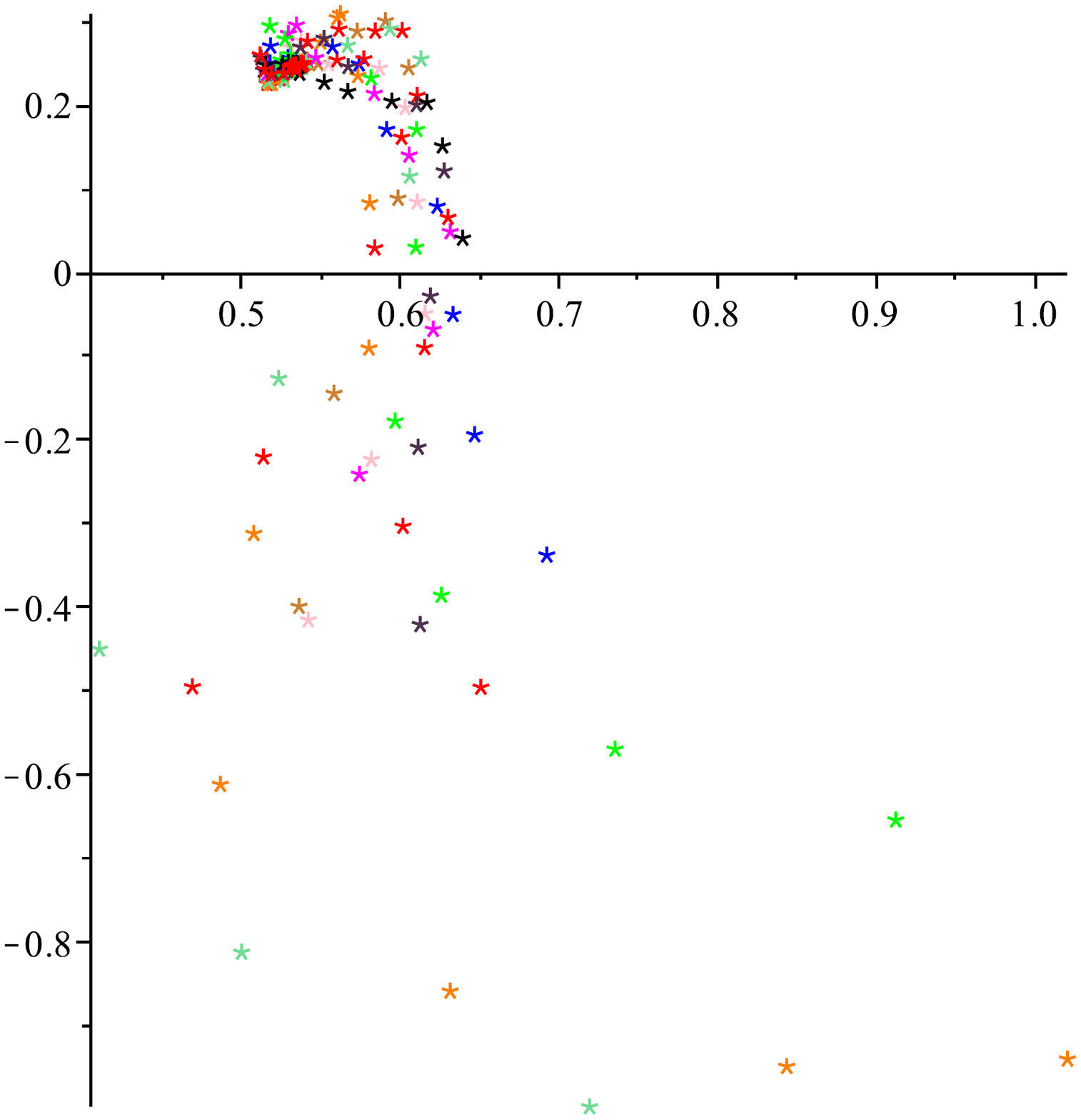}
}
\vspace{1cm}
\centerline{
Fig. 9 Ikeda chaos \hspace{1.5cm} Fig. 10 N=2 \hspace{1.5cm} Fig. 11 N=8
}

\centerline{
\includegraphics[scale=0.15]{{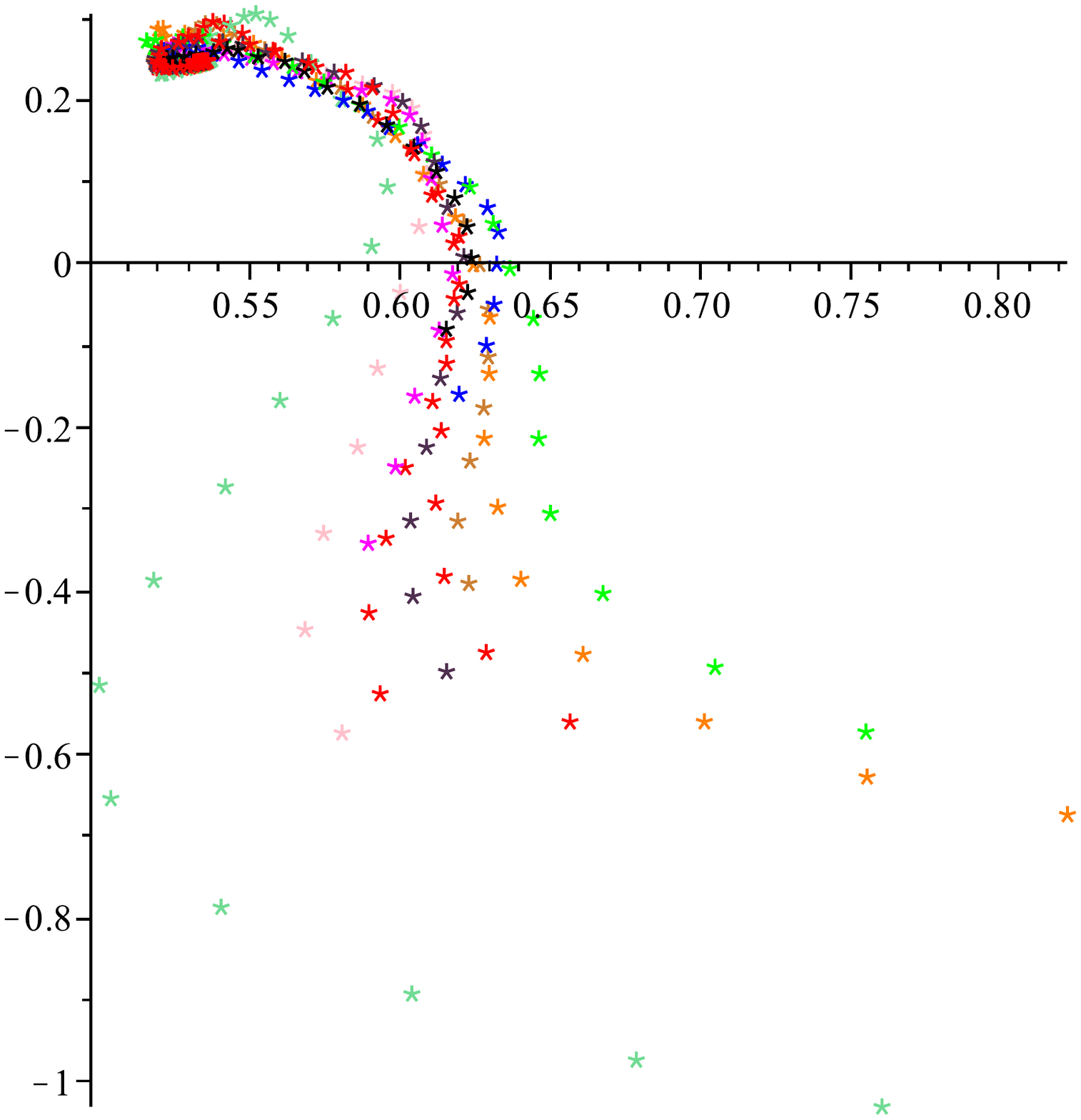}}\hspace{1.5cm}
\includegraphics[scale=0.15]{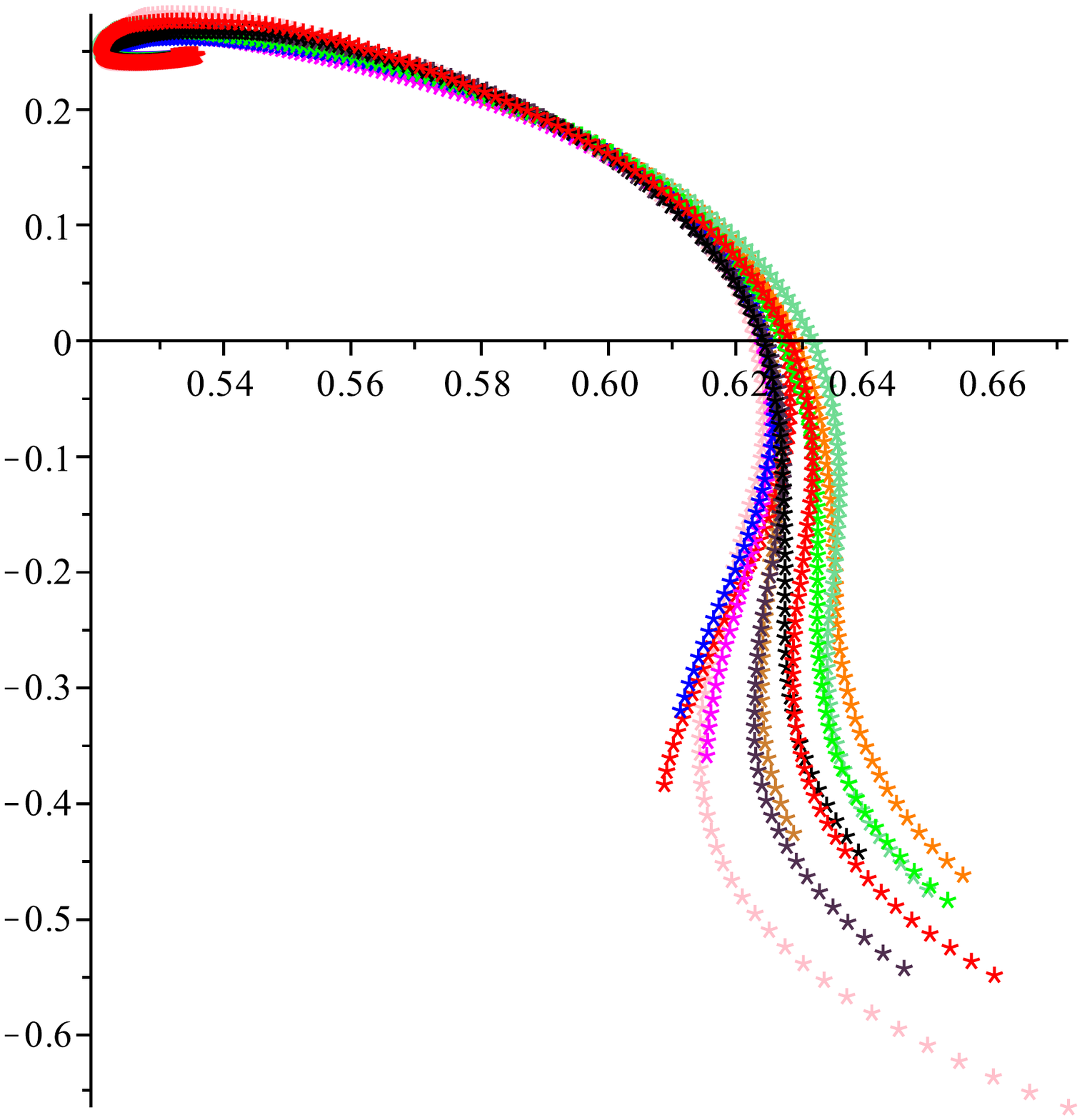}
}
\centerline{
Fig. 12 N=20 \hspace{1.5cm} Fig. 13 N=150 }

The displayed images suggests that any stabilization went through 5 basic stages: Chaos (Fig. 9, 14), 
Fluctuation (Fig. 10, 15), 
Separation (Fig. 11, 16), 
Concentration (Fig. 12, 17) 
and Stabilization (Fig. 13, 18). 

Another example is a famous Arnold cat map $(x, y) \to ( x + y , x+2y) mod1$ \cite{A} which is a classical case of Anosov diffeomorphism. 
 We come up with a remarkable discovery - even for a non-smooth mapping the addition of the control makes a close-loop system more structured. 

\centerline{
\includegraphics[scale=0.15]{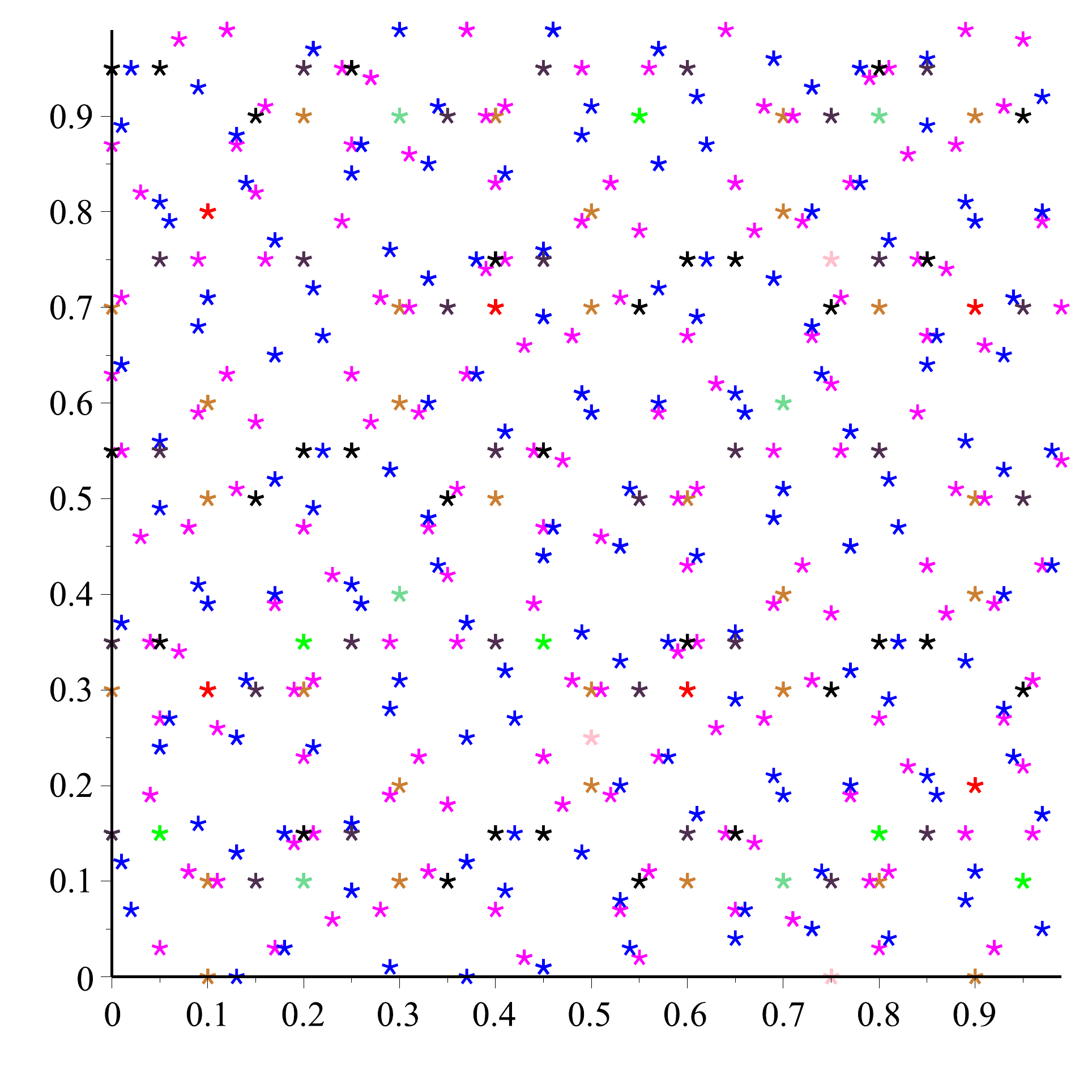}\hspace{1.5cm}
\includegraphics[scale=0.15]{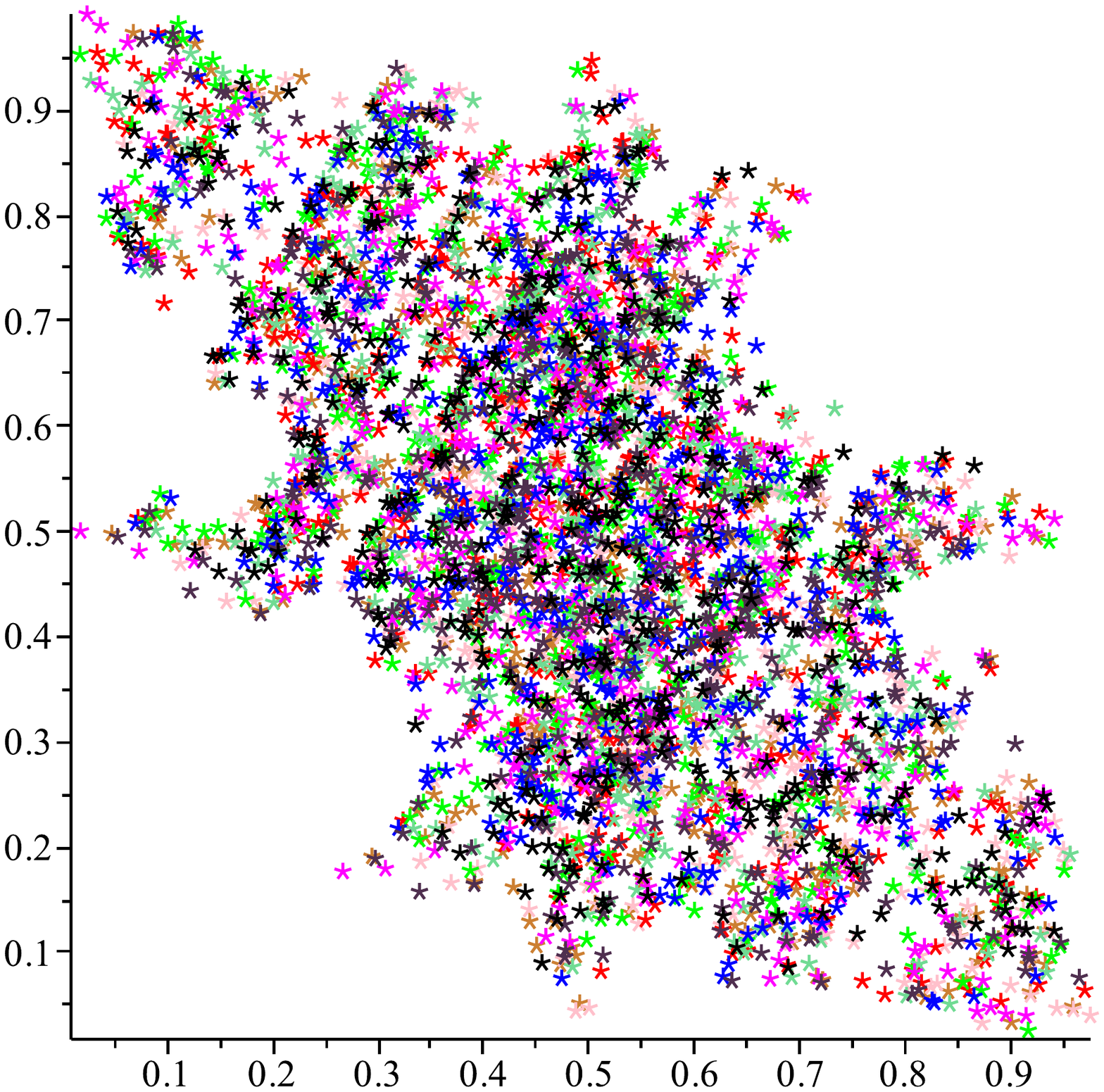}\hspace{1.5cm}
\includegraphics[scale=0.15]{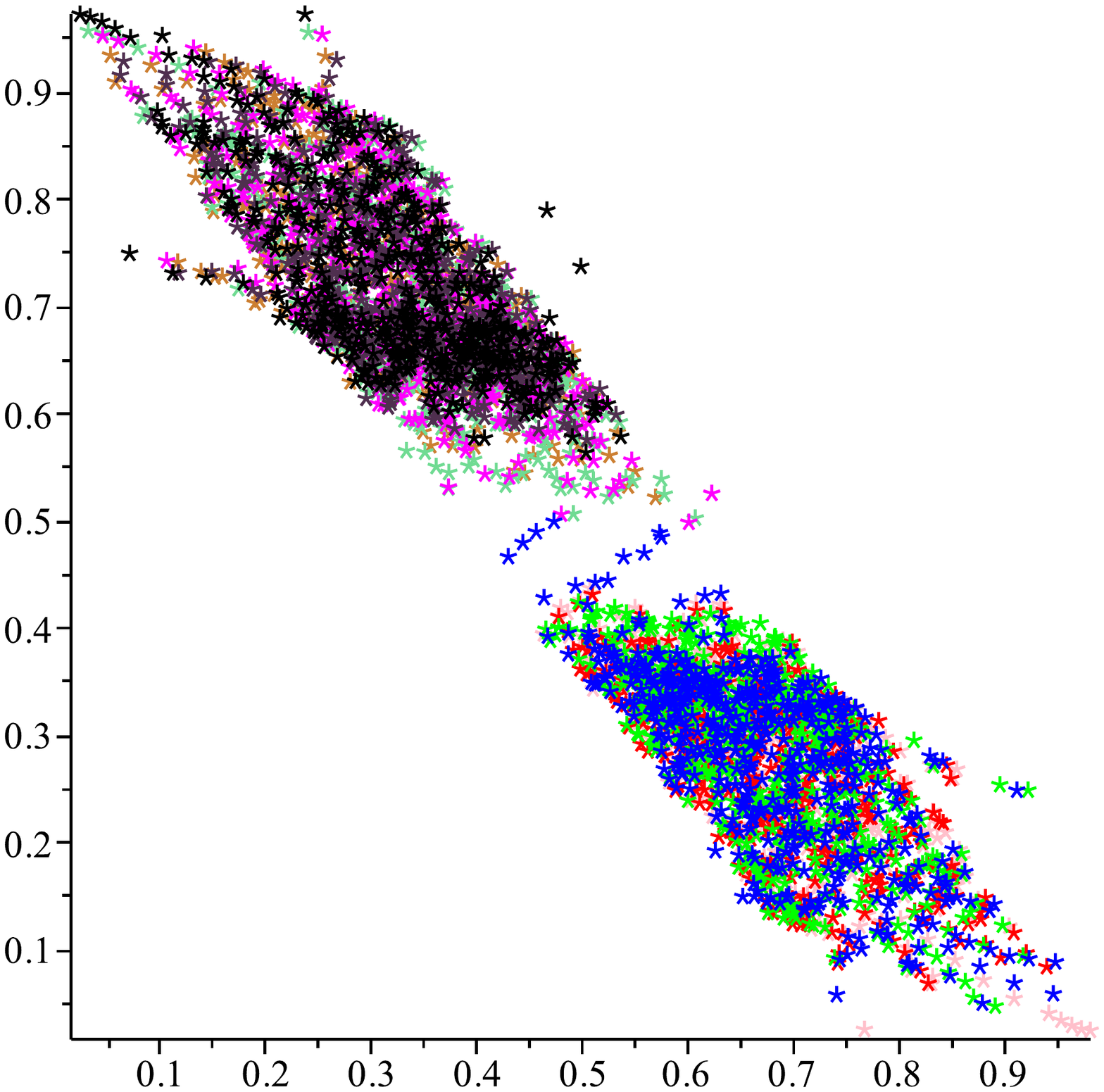}
}
\centerline{
Fig. 14 Arnold cat \hspace{1cm} Fig. 15 N=3 \hspace{2cm} Fig. 16 N=9
}

\centerline{
\includegraphics[scale=0.15]{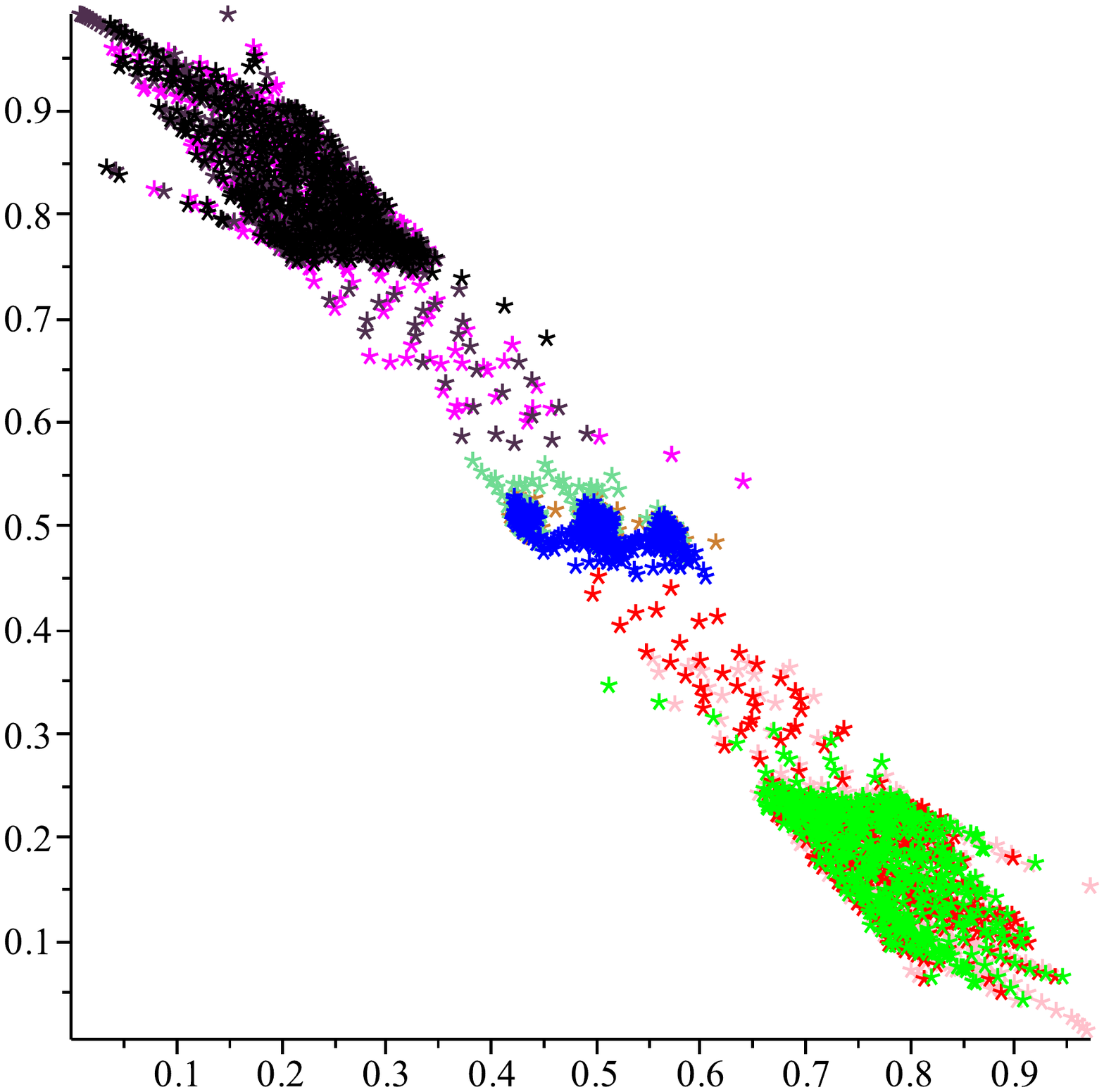}\hspace{2cm}
\includegraphics[scale=0.15]{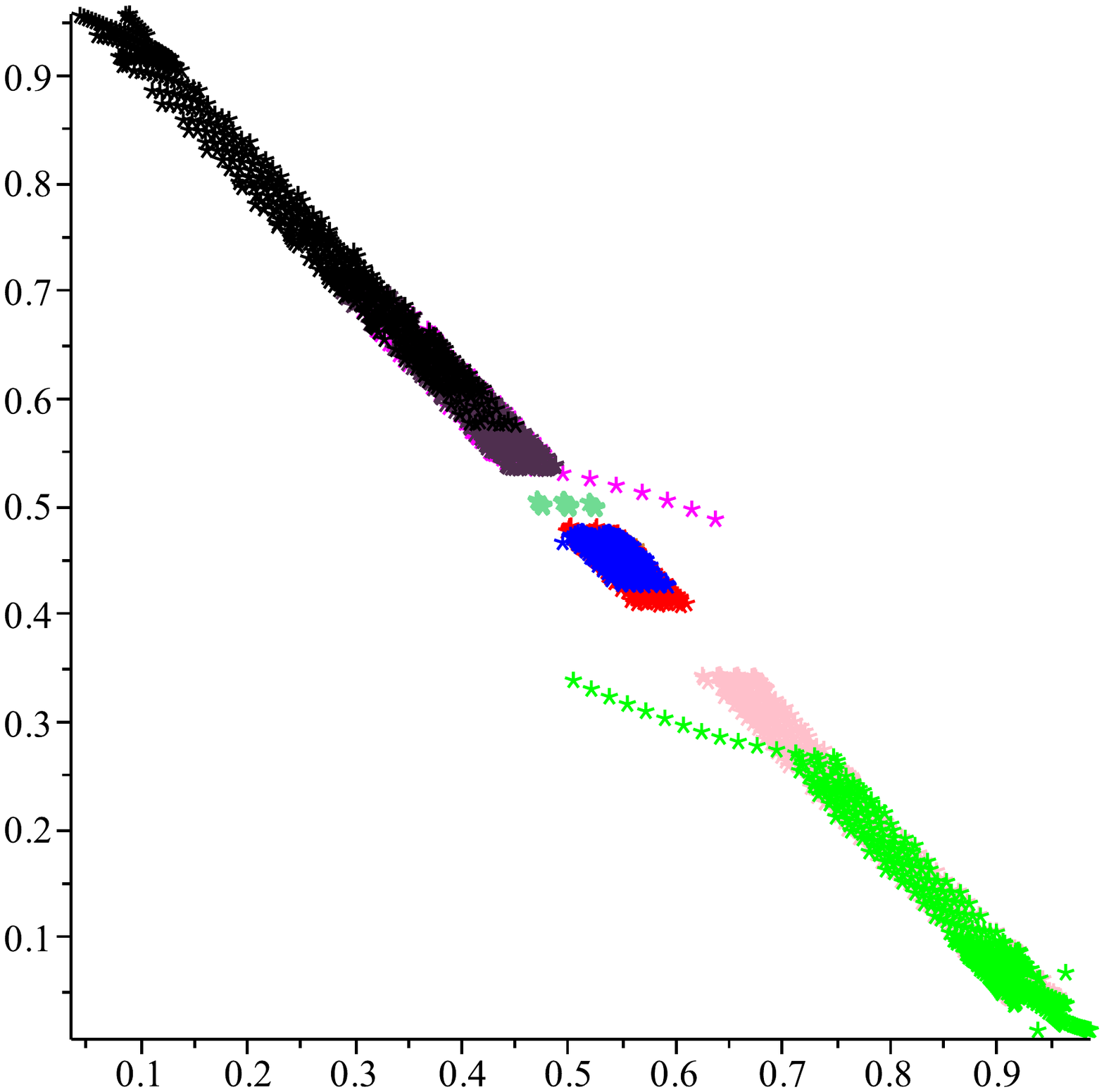}
}
5\vspace{1cm}
\centerline{
\hspace{1cm} Fig. 17 N=15 \hspace{2cm} Fig. 18 N=50 \hspace{1.5cm} 
}

Below are examples of the stabilization of 3D neural sine map\\

\centerline{
$x\to 12\sin(\pi(y-x)),$\\
$y\to 12\sin(\pi(z-y)),$\\
$z\to 12\sin(\pi(x-z))$
}

\begin{figure}[!htb]
\minipage{0.3\textwidth}
\includegraphics[scale=0.15]{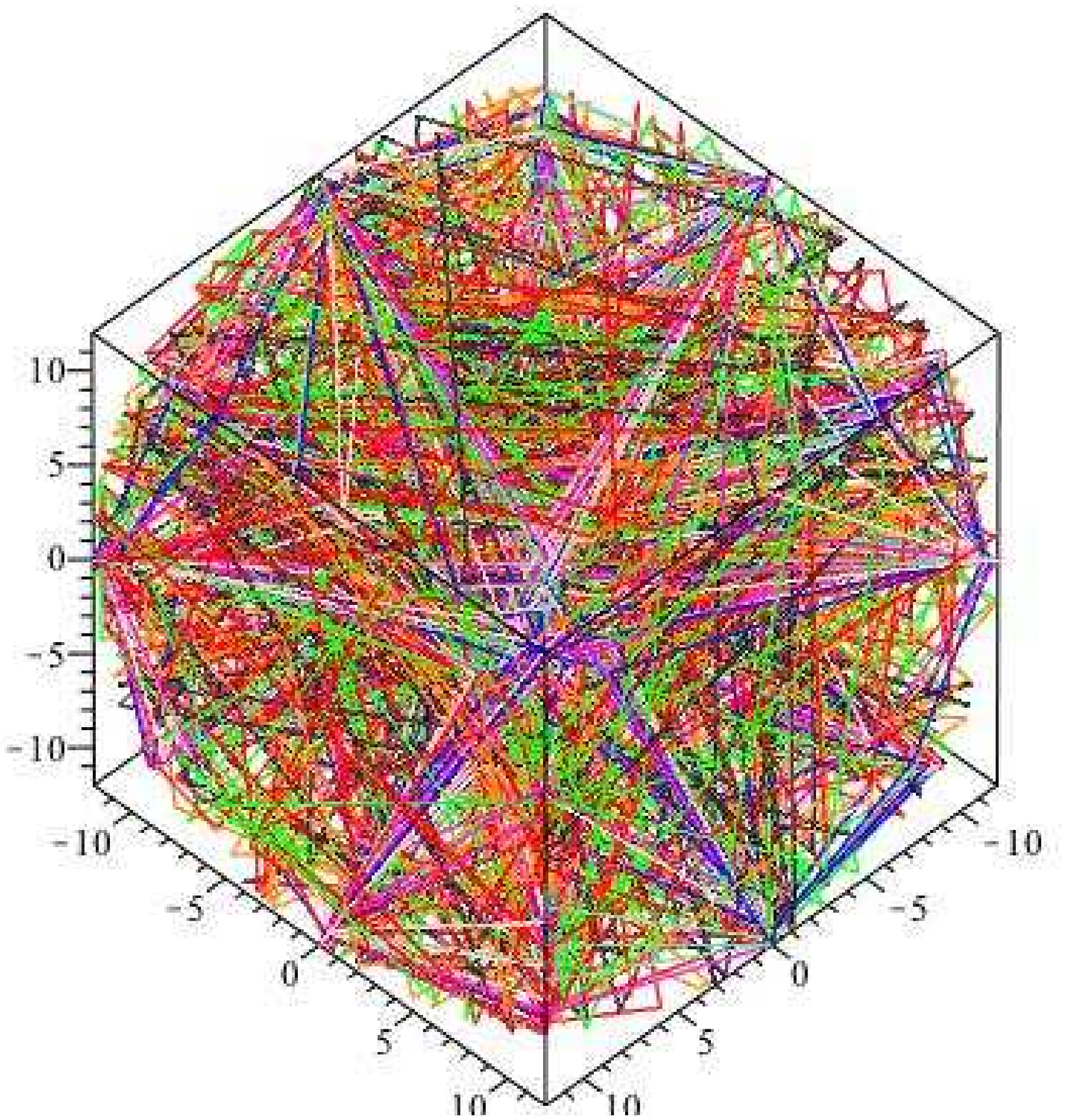}
\endminipage\hfill
\minipage{0.3\textwidth}
\includegraphics[scale=0.15]{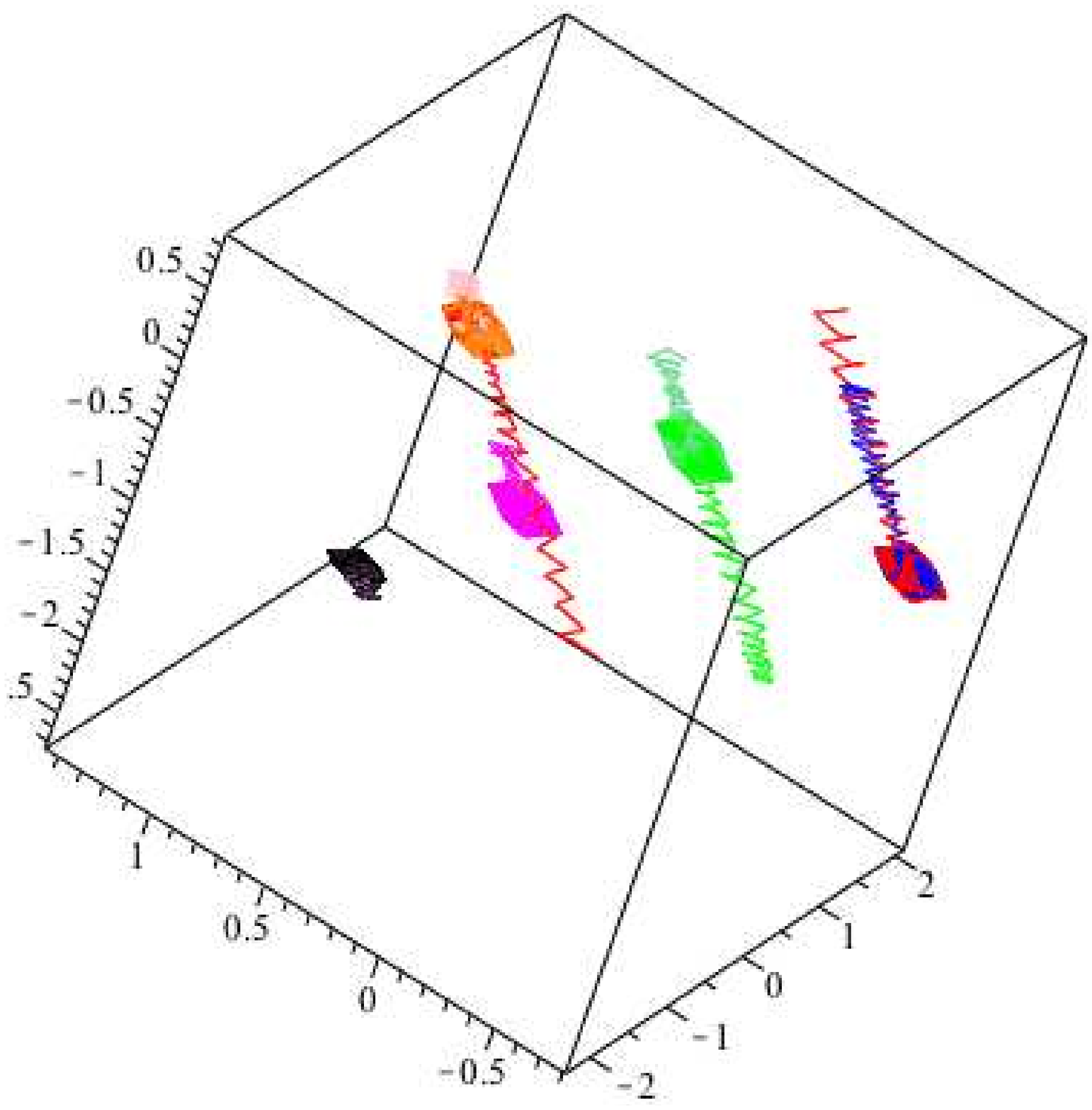}
\endminipage\hfill
\minipage{0.3\textwidth}%
\includegraphics[scale=0.15]{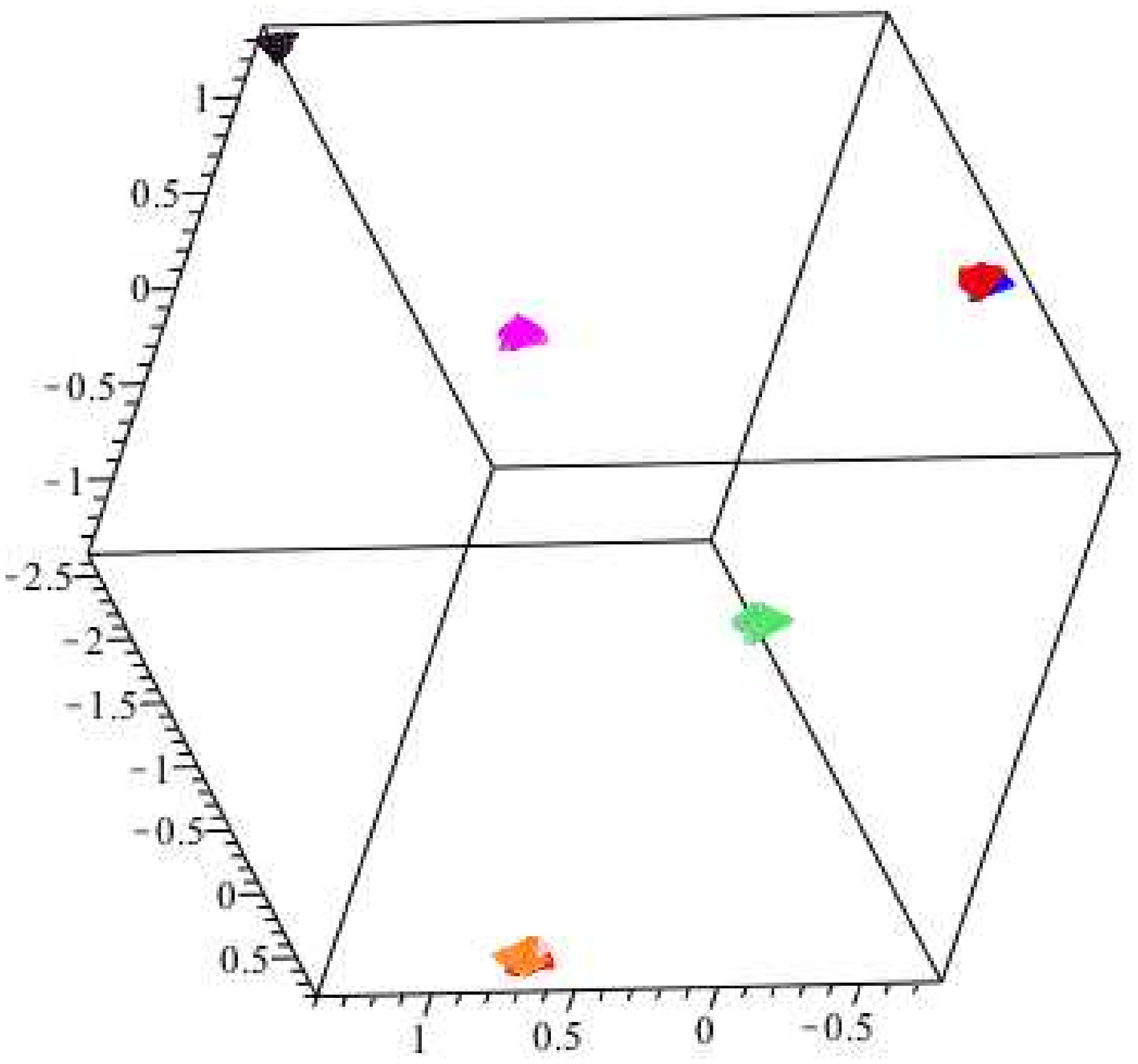}
\endminipage
\end{figure}
\bigskip\bigskip

\centerline{\bf Acknowledgement}
\medskip

 The author would like to express their deep gratitude to Alexei Solyanik for his generosity in supply of ideas and to Pietro Poggi-Corradini for useful discussions and interests to this article.

Address: D.Dmitishin, Odessa National Polytechnic University, 1 Shevchenko Ave., Odessa 65044, Ukraine.

E-mail:{\tt dmitrishin@opu.ua }

A.Khamitova and A.Stokolos, Georgia Southern University, Statesboro, GA 30460, USA

E-mail:{\tt astokolos@Georgiasouthern.edu,\\ \hspace*{1.4cm} Anna\_Khamitova@Georgiasouthern.edu } 

\end{document}